\documentclass[11pt,leqno]{article}
\usepackage{amssymb,amsfonts}%amsLatex
\usepackage{amsmath,amsthm,amsxtra}
\usepackage[all]{xy}

\usepackage[linktocpage=true, colorlinks=true, linkcolor=blue, citecolor=red, urlcolor=green]{hyperref}

\newcommand{\st}[1]{\ensuremath{^{\scriptstyle \textrm{#1}}}}

\newcommand\bigcheck[1]{#1 \raise1ex\hbox{$\hspace{-1ex}{}^\vee$}}
\newcommand\sucheck[1]{#1 \raise0.5ex\hbox{$\hspace{-1ex}{}^\vee$}}

\newcommand{\ad}{\mathop{\rm ad}\,}

\newcommand{\ch}{{\rm ch}}
\newcommand{\rank}{{\rm rank}}

\newcommand{\End}{\mathop{\rm End }}

\renewcommand{\Im}{\mathop{\rm Im  \, }}

\renewcommand{\ne}{\mathop{\rm ne}\,}

\renewcommand{\sl}{s\ell}
\newcommand{\str}{{\rm str}}

\newcommand{\sdim}{\mathop{\rm sdim \, }}

\newcommand{\C}{\mathcal{C}}

\newcommand{\FF}{\mathbb{F}}

\newcommand{\ZZ}{\mathbb{Z}}

\newcommand{\fa}{\mathfrak{a}}

\newcommand{\fg}{\mathfrak{g}}
\newcommand{\fh}{\mathfrak{h}}

\makeatletter
\renewcommand\section{\@startsection {section}{1}{\z@}%
                                   {-3.5ex \@plus -1ex \@minus -.2ex}%
                                   {2.3ex \@plus.2ex}%
                                   {\normalfont\large\bfseries}}
\renewcommand\subsection{\@startsection{subsection}{2}{\z@}%
                                     {-3.25ex\@plus -1ex \@minus -.2ex}%
                                     {0ex \@plus .0ex}%
                                     {\normalfont\normalsize\bfseries}}

\@addtoreset{equation}{section}
\makeatother

\newtheorem{theorem}{Theorem}[section]

\newtheorem{lemma}{Lemma}[section]
\newtheorem{corollary}{Corollary}[section]
\newtheorem{proposition}{Proposition}[section]
\newtheorem*{lemma*}{Lemma}

\theoremstyle{remark}
\newtheorem{remark}{Remark}[section]

\def\be{\ifnum \count1=0 $$ \else \begin{equation}
\fi}
\def\ee{\ifnum\count1=0 $$ \else
\end{equation}\fi}
\def\ele(#1){\ifnum\count1=0 \eqno({\bf #1}) $$
\else \label{#1}\end{equation}\fi}
\def\req(#1){\ifnum\count1=0 {\bf #1}\else
\ref{#1}\fi}
\def\bea(#1){\ifnum \count1=0   $$
\begin{array}{#1}
\else \begin{equation} \begin{array}{#1} \fi}
\def\eea{\ifnum \count1=0 \end{array} $$
\else  \end{array}\end{equation}\fi}
\def\elea(#1){\ifnum \count1=0 \end{array}
\label{#1}\eqno({\bf #1}) $$
\else\end{array}\label{#1}\end{equation}\fi}
\def\cit(#1){
\ifnum\count1=0 {\bf #1} \cite{#1} \else
\cite{#1}\fi}
\def\bibit(#1){\ifnum\count1=0 \bibitem{#1} [#1    ] \else \bibitem{#1}\fi}

\makeatletter
\def\@maketitle{\newpage
 \null
 \vskip 2em
 \begin{center}%
  \vskip 3em
  {\Large\bf \@title \par}%
  \vskip 1.5em
  {\normalsize
   \lineskip .5em
   \begin{tabular}[t]{c}\@author
   \end{tabular}\par}%
  \vskip 2em

 \end{center}%
 \par
 \vskip 2.5em}
\makeatother
\begin{document}

\title{On free field realization of quantum affine W-algebras}

 \author{Victor G. Kac\thanks{Department of Mathematics, M.I.T.,
     Cambridge, MA 02139, USA.~~kac@math.mit.edu}~~\thanks{Supported in part
  by Bert and Ann Kostant fund and by Simons collaboration grant.}~~and Minoru Wakimoto\thanks{~~wakimoto@r6.dion.ne.jp}
   %~~\thanks{Supported by Grant-in-aid 13440012 for
    % scientific research Japan.}
   \\ \text{}}

\maketitle

 \begin{abstract}
  We find an explicit formula for the conformal vector
  of any quantum affine $W$-algebra in its free field realization.
 % obtained from affine superalgebras by quantum
  % reduction.  As an application, we obtain in a unified way free field
  % realizations and determinant formulas for all superconformal
  % algebras.
 \end{abstract}

\section{Introduction}
\label{sec:intro}

\ \ \ \ The chiral part of a (super)conformal field theory is a vertex algebra
which admits a conformal vector $L$, for which the eigenvalues of the energy operator $L_0$ lie in $\frac{1}{2}\ZZ_{\geq 0}$ and the multiplicity of the 0 eigenvalue is 1. An important class of such vertex algebras are
%Recall the construction of a
quantum affine W-algebras $W^{k}(\fg , x, f)$ \ \cite{KRW} , \cite{KW} (see also \cite{DSK}),
%.  This is a freely generated vertex algebra,
attached to a ``good''  datum
$(\fg , x , f, k)$
, where
$\fg = \fg_{\bar{0}} \oplus \fg_{\bar{1}}$ is a basic Lie superalgebra, i. e. a simple finite-dimensional Lie superalgebra over an algebraically closed field $\FF$ of characteristic 0
%\textcolonmonetary \
with reductive even part $\fg_{\bar{0}}$ and a fixed non-degenerate even invariant supersymmetric bilinear form $(. \, | \, .)$, $x\in \fg_{\bar{0}}$ is such that the eigenspace decomposition of $\fg$ with respect to $\ad  x$
defines a
$\tfrac{1}{2} \ZZ$-grading
\begin{equation}
  \label{eq:1.1}
  \fg \ = \ \underset{j \in\tfrac{1}{2} \ZZ}{\bigoplus} \ \fg_{j},
\end{equation}
$f\in \fg_{-1}$, and $k \in \FF$.

A datum $(\fg , x , f, k)$ is called {\it good} if $f\in \fg_{-1}$ is such that
%for the centralizer $\fg^{f}$ of $f$ in $\fg$ one has
\begin{equation}
  \label{eq:1.2}
  \fg^{f} \subset \fg^{}_{\leq 0}.
\end{equation}
Hereafter $\fg^f$ (resp. $\fg_j^f$) denotes the centralizer of $f$ in $\fg$ (resp. $\fg_j$),  
and we use notation $\fg^{}_{\leq m} = \underset{\rm j \leq m}{\oplus} \ \fg_{j}$, and similarly for $\geq m$, or $< m$, or $> m$. We also denote by $p_{> 0}$, $p_{j}$, etc., the projection of $\fg$ to $\fg_{> 0}$, $\fg_{j}$, etc., along (\ref{eq:1.1}).
A special case of a good datum is a {\it Dynkin datum}, defined by $x$ and $f$ from an $\sl_2$-triple $\{ e,x,f\}$ in $\fg_{\bar{0}}$, where $[x,e] = e,[e , f] = x ,[x ,f] = -f$.

%\textcolonmonetary.

%\vspace{5mm}

Recall that a bilinear form $(. \, | \, .)$ on $\fg$ is called even if
$(\fg_{\bar{0}}  |  \fg_{\bar{1}}) = 0$, supersymmetric (resp. superskewsymmetric) if $(a  |   b) = (-1)^{p(a)p(b)} (b  |  a)$ \ (resp. $-(-1)^{p(a)p(b)} (b | a))$, and invariant if $([a,b]  |  c) = (a  |  [b,c])$.

In \cite{KRW} for an arbitrary datum $(\fg , x, f ,k)$ a vertex algebra homology complex
\begin{equation}
  \label{eq:1.3}
  (V^{k} (\fg)\otimes F^{\ch}\otimes F^{\ne}
     \, , \, d^{}_{(0)}) \, 
\end{equation}
was constructed, where $V^{k} (\fg)$ is the universal affine vertex algebra of level $k$ associated to $\fg$, and $F^{\ch}$ (resp. $F^{\ne})$ is the vertex algebra of free charged fermions based on $\fg^{}_{> 0} \oplus \fg^{*}_{> 0}$ with reversed parity (resp. of free neutral fermions based on $\fg_{1/2}$), and $d_{(0)}$ is an explicitly constructed odd derivation of the vertex algebra $\C^{k} (\fg,x,f): = V^{k} (\fg)\otimes F^{\ch}\otimes F^{\ne}$.

Recall \cite{K2} that for the construction of the vertex algebra of free fermions based on a vector superspace $A$, one needs a superskewsymmetric bilinear form on $A$. In the case of $F^{\ch}$ this bilinear form is defined via the pairing of $\fg_{>0}$ and its dual $\fg^{*}_{> 0}$, which is identified with
$\fg_{<0}$, using the bilinear form $(.|.)$; the former is non-degenerate since the latter is. In the case of $F^{\ne}$ this bilinear form is defined by the formula
\begin{equation}
  \label{eq:1.4}
  <a ,b>^{\ne} = (f  |  [a,b]), \ a ,b \in \fg^{}_{1/2}.
\end{equation}
The bilinear form (\ref{eq:1.4}) is non-degenerate if and only if
%we have 
\begin{equation}
  \label{eq:1.5}
  \ad f : \fg_{1/2} \to \fg_{-1/2} \,\, \hbox{is a vector superspace isomorphism}.
\end{equation}  
%The pair $(x,f)$ is called {\it good} if (\ref{eq:1.5}) holds.
Since property (\ref{eq:1.2}) is equivalent to $[f,\fg_j]=\fg_{j-1}$ if $j\leq 1/2$ \cite{KW},
%the pair $(x,f)$ is good
property (\ref{eq:1.5}) holds
for any good datum.

The $\mathbb{Z}$-grading of the complex (\ref{eq:1.3}) is defined by
\begin{displaymath}
 % \label{eq:1.6}
  \deg V^{k} (\fg) = \deg F^{\ne} = 0 , \ \deg \fg^{}_{>0} = \ -\deg \fg^{*}_{>0} = 1.
\end{displaymath}
The homology of the complex (\ref{eq:1.3}) is called the \emph{quantum affine
  W-algebra}, attached to the datum $(\fg ,x , f, k)$, and is denoted by $W^{k} (\fg ,x, f)$.

%Property (\ref{eq:1.2})

For a good datum, $[\fg_0, f]=\fg_{-1}$, hence the orbit $G_0 (f)$
is Zariski open in $\fg_{-1}$, and therefore the vertex algebra $W^k(\fg, x, f)$ is independent, up to isomorphism, of the choice of $f\in \fg_{-1}$,
  satisfying (\ref{eq:1.2}).
  
The main result of \cite{KW} on the structure of the vertex algebra $W^{k} (\fg ,x, f)$ is Theorem 4.1, which states that for a good datum the $j$\st{th} homology of the complex (\ref{eq:1.3}) is zero if $j \neq 0$, and the 0-th homology is the vertex algebra $W^{k} (\fg ,x, f)$, which is a subalgebra of the vertex algebra $\C^{k}(\fg,x,f)$ freely generated by $d_{(0)}$-closed elements $J^{\{a_i\}}$, where $a_1 , \ldots , a_{s}$ is a basis of $\fg^{f}$ consisting of eigenvectors of $\ad  x$.  The elements $J^{\{a_i\}}$ can be recursively computed, using equations (4.11) and (4.12) from \cite{KW}.
The "building blocks" for construction of elements $J^{\{a_i\}}$ are the elements $J^{(a)}$, $a \in \fg^{f}$, defined in \cite{KRW}, see formula
(\ref{eq:2.7})
in Section 2 of the present paper.  Theorem 4.1(a) from \cite{KW} states that for each $a \in \fg^{f}_{-j} \ (= \fg^{f} \cap \fg^{}_{-j})$ the element $J^{\{a\}} - J^{(a)}$ lies in the subalgebra of the vertex a;gebra $\C^{k} (\fg ,x, f)$, generated by elements $J^{(b)}$, where $b \in \fg^{}_{-s}$ with $0 \leq s < j$ (recall that $\fg^{f}_{j} \neq 0$ only for $j \leq 0$ by (\ref{eq:1.2})), and by the neutral fermions.

Consider the subalgebra
$\overline{\C}^k (\fg ,x, f)$
of the vertex algebra $\C^k (\fg ,x, f)$ generated by the elements $J^{(v)}$ with $v \in \fg_{\leq 0}$, and
by the neutral fermions.  It follows from the above discussion that, for a good datum,  all elements
$J^{\{v\}}$, $v \in \fg^{f}$, lie in $\overline{\C}^k (\fg ,x,f)$.  It is easy to see that the elements $J^{(v)}$ with $v \in \fg^{}_{0}$ and neutral fermions
generate a subalgebra
$\overline{\C}^k_{0} (\fg ,x, f)$
of the vertex algebra
$\C (\fg ,x, f)$
, and that the $J^{(v)}$ with $v \in \fg^{}_{< 0}$ generate an ideal $U^{}_{< 0}$ of $\overline{\C}^k (\fg ,x, f)$  , such that $\overline{\C}^k_0 (\fg ,x, f) \ \cap \ U^{}_{< 0} = 0$.  Hence, the canonical map $\overline{\C}^k (\fg ,x, f) \to
\overline{\C}^k_{0} (\fg ,x,f)$
%$\overline{\C}^{}_0 (\fg ,x)$
induces a vertex algebra homomorphism
\begin{equation}
  \label{eq:1.6}
  W^{k} (\fg ,x, f) \ \xrightarrow \ {\overline{\C}}^k_{0} (\fg ,x,f).
\end{equation}
Since the vertex algebra $\overline{\C}^k_{0} (\fg ,x, f)$ is isomorphic to the tensor product of the universal affine vertex algebra $V^{k'} (\fg^{}_{0})$ of "shifted" level $k'$  (\cite{KW} , formula (2.5)), and the vertex algebra $F^{\ne}$, the map (\ref{eq:1.6}) may be viewed as a free field realization (FFR) of the W-algebra $W^{k} (\fg ,x, f)$.

In the case of a good datum, for the $a^{}_{i} \in \fg^{f}_{j}$ with $j = 0$ or $- 1/2$ the elements $J^{\{a_i\}}$ are uniquely determined by the $a_i$ and they are explicitly constructed in \cite{KW}, Section 2. The construction of these elements is still valid for an arbitrary datum, satisfying property (\ref{eq:1.5}).
%holdsfor which the pair $(x,f)$ is good.
Furthermore, provided that $k \neq -h^{\vee}$ (i.e. $k$ is not the critical level), we also constructed there an energy-momentum element $L$, with respect to which the elements $J^{\{a_i\}}$ have conformal weight $1 - m_i$, where $[x ,a_i] = m^{}_{i} a^{}_{i}$, and this construction is again valid for an arbitrary datum
satisfying property (\ref{eq:1.5}).
%holdswith a good pair $(x,f)$.

In \cite{KW}, Theorem 5.1(c), we found an explicit expression for
$L$ in terms of the elements $J^{(a_i)}$ and the neutral fermions
in the case of minimal W-algebras, which allowed us to compute the FFR
(\ref{eq:1.6}) for these W-algebras explicitly (see \cite{KW}, Theorem  5.2).

The main results of the present paper, valid for an arbitrary datum $(\fg, x,f,k)$ satisfying property (\ref{eq:1.5}),
%with a  good pair $(x,f)$,
are Theorem \ref{th:3.1}, which gives an explicit expression of the element $J^{ \{ f \} }$ in terms of the elements
$J^{(f)}$, $J^{(a)}$ with $a\in \fg_0$ and $\fg_{-1/2}$, and of neutral free fermions, and Theorem \ref{th:3.2}, which states that $L=-\frac{1}{k+h^\vee}J^{ \{ f \} }$,
%when (\ref{eq:1.1}) is a Dynkin grading
%and the elements of the subalgebra of $\hat{\C^{k}} (\fg ,x)$, generated by $J{ai}$ with $mi =0$ or $\frac{1}{2}$ , and neutral fermions
provided that $k \neq - h^{\vee}$. This leads to an explicit formula for the image of $L$ under the FFR (\ref{eq:1.6}) for an arbitrary quantum affine W-algebra, attached to a good datum.

Throughout the paper the base field $\FF$ is an algebraically closed field of characteristic 0.

%We would like to thank A. Arakawa, L. Han, and D. Panyushev for very useful correspondence.
\vspace{5mm}

\section{ {The complex $(\C^k(\fg, x, f), \,d_{(0)})$ and the W-algebra $W^k (\fg ,x, f)$.}}
\label{sec:2}

First, recall the construction of vertex algebras $V^k (\fg)$, $F^{\ch}$ and
$F^{\ne}$. We shall use the very convenient language of non-linear Lie conformal superalgebras and $\lambda$-brackets \cite{DSK}. 

Given a Lie superalgebra $\fg$ with an invariant supersymmetric bilinear form B, consider the
%\textcolonmonetary
$\FF [\partial]$-module
%\textcolonmonetary
$\FF [\partial] \otimes \fg$ with the following non-linear $\lambda$-bracket
\begin{equation}
  \label{eq:2.1}
  [a^{}_{\lambda} b] = [a ,b] + \lambda B (a ,b) 1 , \hspace{4mm} a ,b \in \fg ,
\end{equation}

\noindent and the universal enveloping vertex algebra $V^{B} (\fg)$ of this non-linear Lie conformal superalgebra.  One often fixes such a bilinear form
$(.  |  .)$, lets $B (a ,b) = k (a  |  b)$, $k \in \FF$,
%\textcolonmonetary,
and uses the notation
$V^k (\fg) = V^{B} (\fg)$. Then $V^{k} (\fg)$ is called the universal affine vertex algebra for $\fg$ of level $k$.

The vertex algebra $F (A)$ of fermions based on the vector superspace $A$ with a skewsupersymmetric bilinear form $< . , . >$ is defined as the universal enveloping vertex algebra of the
%\textcolonmonetary
$\FF [\partial]$-module
%\tex%tcolonmonetary
$\FF [\partial] \otimes \fg$ with the non-linear $\lambda$-bracket
 \begin{equation}
  \label{eq:2.2}
  [a_{\lambda} b]  =  < a ,  b > 1, \hspace{4mm} a ,b \in A .
\end{equation}

Given a datum $(\fg ,x , f, k)$ as described in the introduction, the associated
homology complex $(\C^{k} (\fg ,x, f) ,d_{(0)})$ is constructed as follows.  Let $A^{\ch} =\Pi( \fg_{> 0} \oplus \fg^{*}_{> 0})$, where $\Pi$ stands for the reversal of parity, and define on it a skewsupersymmetric bilinear form $<.  ,  .>^{\ch}$ by the pairing of the vector superspace $\Pi \fg_{>0}$ and its dual
$\Pi \fg^{*}_{> 0}$, and let
$A^{\ne} = \fg_{1/2}$ with the bilinear form $< a,b >^{\ne}$ defined by (1.4).  
Then $\C^{k} (\fg ,x, f)$ is the universal enveloping vertex algebra of the non-linear Lie confirmal superalgebra
%\textcolonmonetary
$\FF [\partial] (\fg \oplus A^{\ch} \oplus A^{\ne})$ with the $\lambda$-brackets defined by (2.1) and (2.2) on the summands and zero between the distinct summands.  The vertex algebra $\C^k (\fg ,x, f)$ is isomorphic to
$V^k (\fg) \otimes F^{\ch} \otimes F^{\ne}$, where $F^{\ch}=F(A^{\ch})$ and
$F^{\ne}=F(A^{\ne})$.  Letting
  \begin{equation}
     \label{eq:2.3}
     \deg V^k(\fg) = \deg F^{\ne} = 0 , \hspace{4mm} \deg \fg_{> 0} = - \deg \fg^{*}_{> 0} = 1\,\, \hbox{on}\,\, F^{\ch},
 \end{equation}
  \noindent defines a $\mathbb{Z}$-grading of this vertex algebra :
\begin{equation}
    \label{eq:2.4}
    \C^{k} (\fg ,x, f) =
    \underset{j \in \mathbb{Z}}{\oplus} \, \C^{k}_{j}\, .
\end{equation}

In order to define the differential $d_{(0)}$ choose a basis
$\{u_{i} \}^{}_{i \in S}$
of $\fg$, compatible with parity and the $\frac{1}{2} \mathbb{Z}$-grading (1.1), let $\{ u^i \}_{i \in S}$ be its dual basis of $\fg$ with respect to the bilinear form $(.|.)$,  i. e.   $(u_i|u^j)=\delta_{i,j}$,
%let $\{u^i\}_{i\in S}$ be the dual basis with respect to $(.|.)$, i.e. $(u_i|u^j)=\delta_{i,j}$,
and denote by $\{u_i \}^{}_{i \in S_{ > 0}}$ (resp.
$\{ u_{i} \}_{i \in S_{j}}$)
the part of $\{u_i\}_{i\in S}$, which is a basis of $\fg_{>0}$ (resp. $\fg_{j}$). Let $\{\varphi_{i} \}^{}_{i \in S_{ > 0}}$
be the corresponding to $\{u_{i} \}^{}_{i \in S_{>0}}$ basis of $\Pi\fg_{> 0}$, and let $ \{\varphi^{i} \}^{}_{i \in S_{ > 0}}$ be the dual basis of  
$\Pi\fg^{*}_{> 0}$.
  Let $\{\Phi_i\}_{{i \in S}_{1/2}}$ be the corresponding to $\{ u_{i} \}_{i \in S_{1/2}}$ basis of   $A^{\ne}$.
  %and, provided that condition (\ref{eq:1.5}) holds, let $\{\Phi^{i} \}_{i \in S_{1/2}}$ be its dual basis with respect to the bilinear form $<. ,.>^{\ne}$, defined by (\ref{eq:1.4}), i.e.
 % $< \Phi_{i} ,\Phi^{j} >^{\ne} = \delta_{i ,j}$.
  For $u \in \fg$ let
  $\Phi_{u} = \underset{i \in S_{1/2}}{\sum}$
  $\gamma_{i} \Phi_{i}$ (resp. $\varphi_{u} = \underset{i \in S_{> 0}}{\sum} \gamma_{i} \varphi_{i}$)
     if
   $p_{1/2} u  $ (resp. $p_{> 0}u$)
   $=\underset{i}{\sum} \gamma_{i} u_i$, where $\gamma_i\in \FF$.
    % , and let $\varphi^u = \underset{i \in S_{> 0}}{\sum} \gamma_i \varphi^i$ if $p_{<0}u=\underset{i}{\sum} \gamma_{i} u^i$.
  %, and similarly for $\Phi^u$ and $\varphi^u$.

Introduce the following element of the vertex algebra $\C^{k} (\fg ,x, f)$ :
\begin{equation}
    \label{eq:2.5}
    \begin{array}{l}
        d = \underset{i \in S_{> 0}}{\sum} \big((-1)^{p(i)} u_{i} \otimes \varphi^{i} \otimes 1 + (f  |  u_{i}) \otimes \varphi^{i} \otimes 1\big) + \underset{i \in S_{1/2}}{\sum} 1 \otimes \varphi^{i} \otimes \Phi_{i}  \\
        +\frac{1}{2} \underset{i ,j  \in S_{>0}}{\sum} (-1)^{p(i)} 1
     \otimes : \varphi^{i} \varphi^{j} \varphi_{[u_j,u_i]}: \otimes 1,
    \end{array}
\end{equation}
%\[ [u_{i} ,u_{j}] = \underset{s}{\sum}
%\ c^{s}_{ij} \ u_{s}, \]
where $p(i)$ stands for the parity $p(u_{i})$ in the Lie superalgebra $\fg$.
The element $d$ is independent of the choice of the basis of $\fg$. One checks that $[d_{\lambda} d] = 0$
(\cite{KRW}, Theorem 2.1), therefore $[d_{(0)} ,d_{(0)}] = 0$ and $d^{2}_{(0)} = 0$.  Thus, $d_{(0)}$ is a homology differential of the vertex algebra $\C^{k} (\fg ,x, f)$.  The homology of the complex $(\C^{k} (\fg ,x, f), d_{(0)})$ is the quantum affine W-algebra $W^{k} (\fg ,x, f)$.

One has the following formulas for the action of $d_{(0)}$ of the generators of the vertex algebra $\C^{k} (\fg ,x, f)$ (cf. \cite{KRW}, formula(2.4)), where $a \in \fg$, and thereafter we skip the tensor product signs:
\bea(lll)
&d_{(0)} a = &
\underset{j \in S_{> 0}}{\sum}
\big((-1)^{p (j)} :\varphi^j [u_j,a]: + k(a|u_j)\partial \varphi^j\big);\\[1ex]
%(-1)^{p (a)} 
%\partial \varphi^{a}; \\
%
&d_{(0)} \varphi_a = & p_{> 0} (a) + (a|f) + (-1)^{p(a)} 
\Phi_{a} + \underset{j \in S_{> 0}}{\sum} :
\varphi^{j} \varphi_{[u_{j} ,p_{> 0} a]} :; \\[1ex]
& d_{(0)} \varphi^{i} = &- \frac{1}{2}
\underset{j,s \in S_{> 0}}{\sum}
(-1)^{p(i)p(j)}c_{j,s}^i :\varphi^j \varphi^s:, \,\,\hbox{where}\,\,
  [u_j,u_s]=\sum_i c_{j,s}^i u_i \ ; \\[1ex]
& d_{(0)} \Phi_{a} = &\underset{j \in S_{1/2}}{\sum}(u_j|[a,f])
\varphi^{j}.
%; \,\,
%[d_{(0)}\Phi^a] = 
%\varphi^{a} \ .
\elea(eq:2.6)

\iffalse
\begin{eqnarray}
    \label{eq:2.6}
 d_{(0)} a &= & 
\underset{j \in S_{> 0}}{\sum}
(-1)^{p (j)} :\varphi^j [u_j,a]: + k(-1)^{p (j)} 
\partial \varphi^{j};  \\[1ex]
 d_{(0)} \varphi_a &= & 
P_{> 0} a + (a|f) + (-1)^{p(a)} 
\Phi_{a} + \underset{j \in S_{> 0}}{\sum} :
\varphi^{j} \varphi_{[u_{j} ,p_{> 0} a]} : \\[1ex]
 d_{(0)} \varphi^{a} &= &
\frac{1}{2} \underset{j \in S_{> 0}}{\sum}
(-1)^{p(j)} :\varphi^{j}
\varphi^{[u_{j} , a]}: \ ; \\[1ex]
 d_{(0)} \Phi_{a} &= &
\varphi^{[p_{1/2} a , f]}; \,\,
[d_{(0)}\Phi^a] = 
\varphi^{a} \ . \\[1ex]
\end{eqnarray}
\fi
%\vspace{4mm}
Recall that the "building blocks" for elements of the W-algebra $W^{k}(\fg ,x, f)$ are the following elements of $\C^{k}(\fg ,x, f)$ for $v\in \fg$ : 
\begin{equation}
    \label{eq:2.7}
    J^{(v)} = v + \sum_{\substack{j \in S_{> 0}}} (-1)^{p(j)}
    : \varphi_{[v,u_j]}\varphi^j :.
\end{equation}
%where $(c^{i}_{j} (v))$ is the matrix of $\ad v$ in $\fg$ in the basis
%$\{u_{i}\}_{i \in S}$ , i.e. 
%\begin{equation}
 %   \label{eq:2.8}
  %  [v ,u_{j}] = \sum_{\substack{i \in S}} c^{i}_{j} (v) u_{i}.
%\end{equation}

Denote by
$\C^k_- (\fg ,x, f)$
the subalgebra of the vertex algebra
$\C^k (\fg ,x, f)$
, generated by the elements
$J^{(v)}$ ($v\in S_{\leq 0}$), 
$\varphi^i$ ($i\in S_{>0}$), and
$\Phi_i$ ($i\in S_{1/2}$).
By (\ref {eq:2.6}), this subalgebra is $d_{(0)}$-invariant. Let, as above,
$\overline{\C}^k (\fg ,x, f)$
be the subalgebra of $\C^k_- (\fg ,x, f)$, generated
by the $J^{(v)}$ ($v\in S_{\leq 0}$), and the $\Phi_i$ ($i\in S_{1/2}$).
Then, by (\ref{eq:2.6}), we have
\begin{equation}
  \label{eq:2.8}
\overline{\C}^k (\fg ,x, f) \cap d_{(0)}\C^k_- (\fg ,x, f)=0. 
\end{equation}
  
Let $\kappa (a ,b) = str_{\fg} (\ad a)(\ad b)$ be the Killing form on $\fg$.
Recall that
\begin{equation}
  \label{eq:2.8a}
\kappa(a,b)= 2h^\vee(a|b).
  \end{equation}
For the projection $p_j : \fg \to \fg_j$ (resp $p_{>0}:\fg \to \fg_{> 0}$) along the grading (\ref{eq:1.1}), define the "partial" Killing forms 
\begin{equation}
    \label{eq:2.9}
    \kappa_{j} (a ,b) = \str_{\fg} (p_j (\ad a) (\ad b))\,
    (\hbox{resp.}\, \kappa_{> 0} (a ,b) = \str_{\fg} (p_{> 0} (\ad a) (\ad b)).
\end{equation}

Elements $J^{(v)}$ for $v \in \fg_0$ obey $\lambda$-brackets of the universal
affine vertex algebra $V^{B_{0}} (\fg_{0})$ [KRW , Theorem 2.4(c)]:
\begin{equation}
    \label{eq:2.10}
    [{J^{(a)}}_{\lambda} J^{(b)}] = J^{(a ,b)} + \lambda B_{0} (a ,b) , \ a ,b \in \fg_{0},
\end{equation}
where
\begin{equation}
    \label{eq:2.11}
    B_{0} (a ,b) = k(a  |  b) + \frac{1}{2} (\kappa (a ,b) - \kappa_{0} (a ,b)).
\end{equation}
In fact, (\ref{eq:2.10}) holds for $a \in \fg_{i}, \,b \in \fg_{j}$ with $ij \geq 0$ (of course $B_{0} (a ,b) = 0$ if $ij\geq 0$ and $i$ or $j$ is non-zero).  Hence, we have the following 

\begin{corollary}
  \label{cor:2.1}
 The subalgebra $V^{B_{0}} (\fg_{< 0})$ of the vertex algebra $V^{B_{0}} (\fg_{\leq 0})$ is an ideal, the factor algebra being
$V^{B_{0}} (\fg_{0})$.
\end{corollary}
\vspace{4mm}
%\vspace{4mm}
%In order to write down the formula for $d_{(0)} J^{(a)}$ it is convenient to
%introduce notation for $u\in \fg$:
%\begin{equation}
%    \label{eq:2.12}
%    \Phi_{u} = \sum_{\substack{i \in S_{1/2}}} \gamma_{i} \Phi_{i} \ \hbox{if}
%    \ p_{1/2} (u) = \sum_{\substack{i \in S_{1/2}}} \gamma_{i} u_{i} \ , \ \ga%mma_{i} \in \FF .
%    textcolonmonetary.
%\end{equation}

The proof of the following formula from \cite{KW}, formula (2.6),
uses formulas
(\ref{eq:2.6}) :
\begin{eqnarray}
    \label{eq:2.13}
    &\,\,\,\,\,\,\, d_{(0)} (J^{(v)}) =&
    %(-1)^{p(v)} \varphi^{[f,v]}
    \sum_{j \in S_{>0}} ([f,v]|u_j)    \varphi^j
    + \sum_{ j \in S_{>0}}
    (-1)^{p(v)(p(j)+1)}
    :\varphi^{j} \Phi_{p_{1/2}[v,u_{j}]}:\\
\nonumber
    && -\sum_{\substack{j \in S_{>0}\\ [v,u_j] \in \fg_{\leq 0}}}
    (-1)^{p(j) (p(v)+1)} :
    \varphi^{j} J^{([v,u_{j}])} :
    + \sum_{j \in S_{>0}} (k(v|u_{j})
    +\kappa_{>0} ( v, u_j))
    \partial \varphi^j  .
\end{eqnarray}

From now on we shall assume that condition (\ref{eq:1.5}) holds, so that we can define the basis
$\{\Phi^i\}_{i\in S_{1/2}}$
of $A^{\ne}$, dual to $\{\Phi_i\}_{i\in S_{1/2}}$ with respect to the bilinear form (\ref{eq:1.4}). Then we have:
$d_{(0)}\Phi^i=\varphi^i$.

As has been mentioned in the introduction, for a good grading the $d_{(0)}$-closed elements
  $J^{\{ a \}}$ are uniquely determined for $a \in \fg^{f}_{j}$ for $j = 0$ or $- 1/2$.  The $d_{(0)}$-closed elements $J^{\{a \}}$ for $a \in \fg^{f}_{0}$
  can be constructed, provided that (\ref{eq:1.5}) holds, and they are as follows (see \cite{KRW}, Theorem 2.4(a)) :
\begin{equation}
    \label{eq:2.14}
    J^{\{ a \}} = J^{(a)} \ + \frac{(-1)^{p (a)}}{2} 
    \sum_{\substack{j \in S_{1/2}}} : 
    \Phi^j \Phi_{[u_j ,a]} :.
\end{equation}
These elements obey $\lambda$-brackets of the universal affine vertex algebra $V^{B_{1/2}} (\fg^{f}_{0})$:
\begin{equation}
    \label{eq:2.15}
    [{J^{\{a \}}}_{\lambda} J^{\{ b \}}] = J^{\{[a ,b] \}} + \lambda B_{1/2} (a ,b),
\end{equation}
where
\begin{equation}
    \label{eq:2.16}
    B_{1/2} (a,b) = k (a|b) + \kappa_{> 0} (a,b) - \frac{1}{2}
    \kappa_{1/2} (a,b).
\end{equation}

%\vspace{4mm}

The $d_{(0)}$-closed elements $J^{\{ v \}}$ for $v \in \fg^{f}_{-1/2}$ are as follows (see \cite{KW} , Theorem 2.1(d)]):
\begin{equation}
    \label{eq:2.17}
    J^{\{v \}} = J^{(v)} - \frac{(-1)^{p(v)}}{3} \sum_{\substack{i ,j \in
        S_{1/2}}} :\Phi^{i} \Phi^{j} \Phi_{[u_j,[u_{i},v]]}:
  +\underset{i \in S_{1/2}}{\sum}\big(:J^{([v, u_i])}\Phi^i:
  %-\underset{i \in S_{1/2}}{\sum}
  -(k(v|u_i)+\kappa_{>0}(v,u_i))\partial\Phi^i\big),  
\end{equation}
and one has (\cite{KW}, Theorem 2.1(e)) :
\begin{equation}
    \label{eq:2.18}
    [ {J^{\{ a \}}}_{\lambda}  J^{\{v \}}] =  J^{\{[a,v] \}} \ \ \hbox{if} \ \ a \in \fg^{f}_{0}, \,v \in \fg^{f}_{- 1/2}.
\end{equation}

\begin{remark}
  \label{rem:2.1}
  The elements
  $\varphi^i$
  coincide with the elements, denoted by $\varphi_i^*$
  in \cite{KRW} and \cite{KW}, but are different from the elements, denoted by $\varphi^i$ in \cite{DSK}. The advantage of this less natural choice is that
  then the construction of the $W$-algebra $W^k(\fg,x,f)$ works for an arbitrary finite-dimensional Lie superalgebra $\fg$ with an arbitrary supersymmetric
  (possibly degenerate) invariant bilinear form $(.|.)$. (The simplicity of $\fg$ and the non-degeneracy of $(.|.)$ are needed
in the next sections.)
  %for the construction of the energy-momentum element $L$ in the next section.
 \end{remark}  
\vspace{4mm}
\section{ {A formula for $J^{\{f\}}$
    and the energy-momentum element $L$ of $W^{k}(\fg ,x,f)$}}
\label{sec:3}

Choose a Cartan subalgebra $\fh$ of $\fg_0$, containing a Cartan subalgebra
%$\fh^f$
of $\fg_0^f$. It is a Cartan subalgebra of $\fg$. Choose a set of positive roots of $\fg$, compatible with the grading (\ref{eq:1.1}).
Recall that the dual Coxeter number $h\spcheck$ of the simple Lie superalgebra $\fg$ with the given invariant bilinear form $(. \ | \ .)$ is the half of the eigenvalue of the Casimir operator $\sum_{j\in S} u^ju_j$ on $\fg$, and  it is given by the formula
\begin{equation}
    \label{eq:3.1}
    h\spcheck \ = \ (\rho  |  \theta) \ + \ \frac{1}{2} \ (\theta  |  \theta) \ ,
\end{equation}
where
%, for a choice of positive roots of $\fg$ ,
$\theta$ is the highest root and $\rho$ is the half of the difference between sums of positive even roots and positive odd roots.
Provided that $k \neq - h\spcheck$, the energy-momentum (or Virasaro) element of the vertex algebra $\C^{k} (\fg ,x, f)$ is defined by \cite{KRW} for an arbitrary datum, satisfying (\ref{eq:1.5}) : 
\begin{equation}
    \label{eq:3.2}
    %\begin{array}{l}
    L \ = \ L^{\fg} \ + \ \partial x \ + \ L^{\ch} \ + \ L^{\ne} ,
    %&  \\
    % &  \\
    %\end{array}
\end{equation}
where
\begin{eqnarray*}
&L^{\fg} =& \frac{1}{2 (k + h\spcheck)} \underset{j \in S}{\sum}  \ : u^{j} u_{j} :  ,   \\
    &L^{\ch}   =&  \underset{j \in S_{> 0}}{\sum} (1 - m_{j}) : (\partial
  \varphi^{j}) \varphi_{j} : \ - \ m_{j} : \varphi^{j} \partial \varphi_{j} : ,  \\
    &L^{ne} =& \frac{1}{2} \underset{j \in S_{1/2}}{\sum} \
    : (\partial \Phi^{j}) \Phi_{j} :  , 
\end{eqnarray*}
\iffalse
\begin{block}
    {\centering
    L^{\fg} = \frac{1}{2 (k + h\spcheck)} \underset{j \in S}{\sum}  \ : u^{j} u_{j} :  , &  \\
    L^{\ch} \ = \ \underset{j \in S_{> 0}}{\sum} (1 - m_{j}) : (\partial
    \varphi^{j}) \varphi_{j} : \ - \ m_{j} : \varphi^{j} \partial \varphi_{j} : , &  \\
    L^{ne} = \frac{1}{2} \underset{j \in S_{1/2}}{\sum} \ : (\partial \Phi^{j}) \Phi_{j} :  , &  \\
     &  \\
     }
\end{block}  
\fi
and the $m_{j}$ are defined by $[x ,u_{j}] = m_{j} u_{j}$.

The central charge of this Virasoro element is equal to (see \cite{KRW}, Remark 2.2) 
\begin{equation}
    \label{eq:3.3}
    c (\fg ,x ,k) \ = \sdim \fg_0 -\frac{1}{2}\sdim \fg_{1/2}-
    \frac{12}{k + h^\vee}|\rho -(k+h^\vee)x|^2.
%    - 12k (x|x) - \underset{j \in S_{> 0}}{\sum} (-1)^{p(j)} (12 m^{2}_{j} - 12 m_{j} + 2) - \frac{1}{2} \sdim \fg_{1/2} .
\end{equation}
With respect to this $L$ the elements $\varphi_{j}$ (resp. $\varphi^{j}$) are primary of conformal weight $1 - m_{j}$ (resp. $m_{j}$), the $\Phi_{j}$ are primary of conformal weight $1/2$, and $a \in \fg_{j}$ has conformal weight $1 - j$ and is primary, unless $j = 0$ and $(a  |  x) \neq 0$. Actually one has:
\begin{equation}
  \label{eq:3.4}
        [L_\lambda a]=(\partial +\lambda)a -\lambda^2 k(a|x)\,\, \hbox{for} \,\,a\in \fg_0.
\end{equation}        

Furthermore, it was shown in \cite{KRW} that the element $d$, defined by $(2.5)$ is primary of conformal weight 1, hence $[d_{\lambda} L] = \lambda d$ and $d_{(0)} L = [d_{\lambda} L] |_{\lambda = 0} = 0$.  Hence, the homology class of $L$ defines an energy-momentum element of the vertex algebra
$W^{k} (\fg ,x, f)$
, which is denoted again by $L$. Note that though, for a good datum, the $W$-algebra $W^{k} (\fg ,x, f)$ is independent, up to isomorphism, of the choice of $x$ with given $f$ \cite{AKM}, the element $L$ does depend on $x$.

As has been mentioned in the introduction, the explicit expressions of the elements $J^{\{a\}}$ which generate the subalgebra $W^{k}(\fg ,x,f)$ of the vertex algebra ${\overline{\C}}^{k} (\fg , x, f)$, associated to a good datum, are known only for $a \in \fg_{-j}$, where $j = 0$ and $\frac{1}{2}$. In view of (\ref{eq:1.6}) it is important to find an explicit expression for $J^{\{ f \}}$.  This is the first main result of the paper.
%In order to state it we need to introduce an additional piece of rotation.
The second main result is the formula $L=-\frac{1}{k+h^\vee}J^{\{f\}}$
in $W^k(\fg,x,f)$. Both results hold for an arbitrary datum $(\fg, x,f,k)$, satisfying (\ref{eq:1.5}).
%  with a good pair $(x,f)$.
%(\ref{eq:1.1}) is a Dynkin grading.

%Recall that the map  $\ad f : \fg_{1} \ \xrightarrow \ \fg_{0}$ is injective , hence the elements $[f ,u_{k}]$ with $k \in S_{1}$ form a basis of the superspace $[f ,\fg_{1}]$.  Moreover, since the restriction of the bilinear form $(.  |  .)$ to $\fg^{f}_0$ is non-degenerate and $\fg^{f}_0 \perp [f ,\fg_{1}]$, it follows from the second formula in (\ref{eq:1.2}) that the restriction of $(.  |  .)$ to $[f ,\fg_{1}]$ is non-degenerate as well.  Let $\{ v^{k} \}_{k \in S_{1}}$ be the basis of $[f ,\fg_{1}]$, dual to the basis $\{ [ f ,u_i ] \}$, i.e.\\

%\vspace{2mm}

%\hspace{50mm} $([f ,u_i]  | v^{j}) \ = \ \delta_{i ,j} , \ i , j \in S_{1}.$

Let
\begin{displaymath}
   % \label{eq:3.4}
    \Omega_{0} \ = \ \underset{j \in S_{0}}{\sum} \ (\ad   u^{j}) (\ad  u_{j}).
\end{displaymath}
\begin{proposition}
  \label{prop:lambda}
  The operator $\Omega_{0}$
 is diagonalizable on $\fg_j$ for each $j> 0$.
\begin{proof}
  Choose a Cartan subalgebra $\fh$ of $\fg_{0}$ ; it is a Cartan subalgebra of $\fg$, containing $x$.  Choose a set of positive roots in $\fh^{*}$, compatible with the $\frac{1}{2} \mathbb{Z} $-grading (\ref {eq:1.1}), and let $e_{i} ,f_{i}$ be the Chevalley generators of $\fg$.  Then for each $j>0$,
  the $\fg$-module
  $\fg_{j}$ is the sum of lowest weight modules with the lowest weight vectors that are commutators of the $e_{i}$, such that $e_{i}\in \fg_{>0}$. Since the restriction of $\Omega_{0}$ to each of these summands is diagonalizable, proposition follows.
\end{proof}
\end{proposition}

Let $\rho_{> 0}$ (resp. $\rho_{j}$)$\in \fh^*=\fh$ be the half of the difference between the sums of positive even and positive odd roots of $\fh$ in
$\fg_{> 0}$ (resp. $\fg_j$). (We idenitfy $\fg$ with $\fg^*$  using $(.|.)$).
\begin{proposition}
  \label{prop:3.2}
 % Suppose that the grading (\ref{eq:1.1}) satisfies (\ref{eq:1.5}). Then
  The element $\rho_{>0}$ lies in the center of $\fg_0$ if (\ref{eq:1.5}) holds.
  \begin{proof}
    By the Jacobi identity for
    $J^{(a)}$, $J^{(b)}$, with $a,b\in \fg_0$, and $L$, we have
\begin{displaymath}
[{J^{(a)}}_\lambda [{J^{(b)}}_\mu L]]-(-1)^{p(a)p(b)}[{J^{(b)}}_\mu [{J^{(a)}}_\lambda L]]=[[{J^{a)}}_\lambda J^{(b)}]_{\lambda +\mu} L].
     \end{displaymath}   
Using (\ref{eq:3.4}) and the skewsymmetry of the $\lambda$-bracket, this gives $0=-(\lambda +\mu)^2(\rho_{>0}|[a,b])$.
Hence $([\rho_{>0},a]|b)=0$ for all $b\in \fg_0$. It follows that
$[\rho_{>0},\fg_0]=0$.
  \end{proof}
  \end{proposition}
%We have the decomposition
%\begin{equation}
 %   \rho_{> 0} = \rho^{\natural}_{> 0} + \rho^{\perp}_{> 0} ,
%\end{equation}
%where $\rho^{\natural}_{> 0}$ (resp. $\rho^{\perp}_{> 0}$ ) is the orthoganal projection of $\fg_{> 0}$ to $\fh^{f}$ (resp. to $\fh^{f \perp}$ in $\fh$).
\begin{theorem}
    \label{th:3.1}
    For
    the datum $(\fg,x,f,k)$, satisfying (\ref{eq:1.5}),
    the following element of $\C^{k} (\fg ,x, f)$ is $d_{(0)}$-closed : 
\begin{displaymath}
J^{\{ f \}} = J^{(f)} + \underset{j \in S_{1/2}}{\sum} (-1)^{p (j)} : \Phi^{j} J^{([f , u_j])} : - \frac{1}{2} \underset{j \in S_0}{\sum} : J^{(u^j)} J^{(u_j)} :
\end{displaymath}
\begin{displaymath}
- (k+h^\vee)\partial J^{(x)}+\partial J^{(\rho_{> 0} )} + \frac{k+h^\vee}{2} \underset{j \in S_{1/2}}{\sum} \ :\Phi^{j} \ \partial \Phi_j: .
%{(k + h^\vee + \Omega_0) u_j -  2 [\rho_{> 0} ,u_j]}. 
\end{displaymath}
\end{theorem}

\begin{theorem} 
    \label{th:3.2}
    Assuming that
    the datum $(\fg,x,f,k)$ satisfies (\ref{eq:1.5}),
    %(\ref{eq:1.1}) is a Dynkin grading
    and that $k \neq -h^\vee$, the element $L + \frac{1}{k + h^\vee}J^{\{f\}}$ of $\C (\fg ,x,f)$ is $d_{(0)}$-exact. Consequently $L=-\frac{1}{k+h^\vee}J^{\{f\}}$ in $W^k(\fg, x,f)$.
\end{theorem}

As an immediate consequence of (\ref{eq:2.10}), (\ref{eq:2.11}),
(\ref{eq:2.14})-(\ref{eq:2.17}) and Theorems
\ref{th:3.1} and \ref{th:3.2}, we obtain the following corollary.
%an injective homomorphism of the vertex algebra $W^k (\fg ,x)$ to the vertex algebra $V^{B_{0}} (\fg_{0}) \otimes F^{ne} (\fg_{\frac{1}{2}})$ , sum that .
\begin{corollary}
  \label{cor:3.1}
  Provided that the datum $(\fg,x,f,k)$ satisfies (\ref{eq:1.5}) and $k\neq -h^\vee$,
  one has a homomorphism of the vertex algebra $W^k (\fg ,x,f)$
  to the vertex algebra $V^{B_{0}} (\fg_0) \otimes F^{\ne} (\fg_{1/2})$, such that 
\begin{displaymath}
J^{\{ a \}} \ \mapsto \ a \ + \ \frac{(-1)^{p(a)}}{2}
\underset{j \in S_{1/2}}{\sum} :\Phi^{j} \Phi_{[u_j ,a]}:\,\,\,
\hbox{if}\,\,\, a \in \fg^{f}_{0} ,
\end{displaymath}
\begin{displaymath}
  J^{\{ v \}} \ \mapsto \underset{i \in S_{1/2}}{\sum}:[v, u_i]\Phi^i: - \frac {(-1)^{p(v)}}{3}
  \underset{i,j \in S_{1/2}}{\sum}
  : \Phi^i \Phi^j \Phi_{[u_j, [u_i ,v]]}:
  \end{displaymath}  
\begin{displaymath}  
  -\underset{i \in S_{1/2}}{\sum}\big(k(v|u_i)+\kappa_{>0}(v,u_i)\big)\partial\Phi^i
  \,\,\, \hbox{if}\,\,\, v\in \fg^f_{-1/2} ,
\end{displaymath}
\begin{displaymath}
L \ \mapsto \ \frac{1}{k + h^\vee} \big(\frac{1}{2} \underset{j \in S_{0}}{\sum} : u^j u_j : \ + \ \partial ((k + h^\vee) x - \rho_{>0})\big) - \frac{1}{2} \underset{j \in S_{1/2}}{\sum} : \Phi^j \partial \Phi_j: .
\end{displaymath}
\end{corollary}

\section{ { Proof of Theorem 3.1}}
\label{sec:4}

Let $U$ and $V$ be finite-dimensional vector spaces over $ \mathbb{F} $ with a non-degenerate even pairing $ <.,.> : U \times V \xrightarrow{} \mathbb{F}$.  Choose dual bases $ \{ u_i \}_{i \in I}$ and $ \{ u^i \}_{i \in I} $ of $U$ and $V$ respectively, i.e. $ < u_i , u^j > = \delta_{ i , j} $. Then for any $ A \in \End U$ and $B \in \End V$ we have :
\begin{equation}
  \label{eq:4.1}
     \str_{U} A = \underset{i \in I}{\sum} (-1)^{p(i)} < A u_i , u^{i} > ,\,\, \str_{V} B = \underset{i \in I}{\sum} (-1)^{p(i)} < u_i , B u^i>,
\end{equation}
where, as before, $p(i)$ stands for $p(u_i) (=p(u^i))$.
This will be used to prove the following lemma.
\begin{lemma}
  \label{lem:4.1} For $u, v \in \fg$ we have
  \begin{equation}
    \label{eq:4.2}
    \kappa_{< 0}  (u,v)  =   \kappa_{> 0} (u,v) - \str_{\fg {> 0}}\, p_{> 0}
    \ad [u,v],
\end{equation}
  \begin{equation}
    \label{eq:4.3}
    \kappa_{> 0}  (u,v)  =   \frac{1}{2} \big(\kappa (u,v) - \kappa_{0} (u,v) \ + \ \str_{\fg > 0} \,p_{> 0} \ad [u,v]\big),
\end{equation}
  \begin{equation}
    \label{eq:4.4}
    \kappa_{ 0}  (u,v)  =   (\Omega_0 u  |  v) = (u  |  \Omega_0  v).
\end{equation}
\end{lemma}
\vspace{4mm}
\begin{proof} We may assume that $(\ad u)(\ad v)$ preserves the $\frac{1}{2} Z $-grading (\ref{eq:1.1}) and that $p(u) = p(v)$.
In order to prove (\ref{eq:4.2}), we use (\ref{eq:4.1}) with 
\begin{displaymath}
    U \ = \ \fg_{> 0} ,\, V = \fg_{< 0} ,\, <.,.> = ( .|. ),\, A = (\ad v)(\ad u), \,B = (\ad u)(\ad v).
\end{displaymath}
We have by the second and then the first formula in (\ref{eq:4.1}):
\begin{displaymath}
    \kappa_{< 0}(u,v) \ = \ \str_{\fg_{< 0}} B = \underset{i \in S_{>0}}{\sum} (-1)^{p(i)} (u_i | [u, [v,u^i] ]) \ = \ \underset{i \in S_{> 0}}{\sum} (-1)^{p(i)} ([ [u_i , u], v ] | u^i )
\end{displaymath}
\begin{displaymath}
  = \ (-1)^{p(u) p(v)} \underset{i \in S_{> 0}}{\sum} (-1)^{p(i)} \
  ([v,  [u, u_i]]  | u^{i}) = (-1)^{p(u) p(v)} \str_{\fg_{> 0}} \ A
\end{displaymath}
\begin{displaymath}
  = \ (-1)^{p(u)p(v)} \str_{\fg_{> 0}} \ \ad [v,u] \ + \ \str_{\fg_{> 0}} \
  (\ad u )( \ad v) = - \str_{\fg_{> 0}} \ad [u,v] \ + \ \kappa_{> 0} (u,v).
\end{displaymath}
Formula (\ref{eq:4.3}) follows from (\ref{eq:4.2}) since 
%\begin{displaymath}
$\kappa (u,v) =  \kappa_{> 0} (u,v)  +  \kappa_0 (u,v) +
\kappa_{< 0} (u,v)$.
%\end{displaymath}
The proof of (\ref{eq:4.4}) is similar, by letting $U = V = \fg_0$, $A = (\ad u)(\ad v)$.
\end{proof}

\begin{lemma}
  \label{lem:4.2} Let, as before, $\{ u_i \}_{i \in S_{> 0}}$ be a basis of
  $\fg_{> 0}$ and $\{ u^i \}_{i \in S_{> 0}}$ the dual basis of $\fg_{< 0}$,
  i.e. $(u_i | u^j) = \delta_{i,j}$, and let $v \in \fg_0$. Then
  \begin{equation}
  \label{eq:4.5}
      \underset{i \in S_{> 0}}{\sum} (-1)^{p(i)} [u_i , u^i]  =  2 \rho_{> 0}, 
  \end{equation}
    \begin{equation}
    \label{eq:4.6}
      \str_{\fg_{> 0}} \ad v  =  2 (\rho _{>0}  |  v).
  \end{equation}
    \begin{proof} Since the LHS is independent on the choice of dual bases, we may take for $\{ u_i \}_{i\in S_{>0}}$ the basis
      $\{ e_{\alpha} \}_{\alpha \in \Delta_{> 0}}$
      of $\fg_{>0}$, so that the dual basis of $\fg_{<0}$ is
$\{ e_{-\alpha} \}_{\alpha \in \Delta_{> 0}}$
      with $ (e_{\alpha}  |  e_{- \alpha}) = 1 $, and hence $[e_{\alpha} , e_{- \alpha} ] = \alpha$. Then (\ref{eq:4.5}) follows.

For (\ref{eq:4.6}) we have :
\begin{displaymath}
  \str_{\fg_{> 0}} \ \ad v \ = \ \underset{i \in S_{> 0}}{\sum} (-1)^{p(i)} ([v , u_i]  |  u^i) \  = \ \underset{i \in S_{> 0}}{\sum} (-1)^{p(i)}
  ( v  |  [u_i , u^i] )  =  2 (v  |  \rho_{> 0})
\end{displaymath}
by (\ref{eq:4.5}).
\end{proof}
\end{lemma}

Denote by I, II, ..., IV the operator $d_{(0)}$, applied to each of the six terms in the RHS of the formula for $J^{\{ f \}}$ in Theorem 3.1.  We have to prove that the sum of these six elements of $\C^{k} (\fg ,x, f)$ is equal to $0$.

By formula (\ref{eq:2.13}) for $v = f$, element I is equal to the sum of the following four elements :
\begin{equation}
\label{eq:4.7}
  I_{A} = \underset{j\in S_{3/2}}{\sum} : \varphi^{j} \Phi_{[f ,u_{i}]} : ,\,\,
  I_{B} = - \underset{j \in S_{1}}{\sum} (-1)^{p(j)} : \varphi^{j} J^{([f ,u_{j}])}:,
\end{equation}
\begin{equation}
  \label{eq:4.8}
  I_{C} = - \underset{j \in S_{1/2}}{\sum} (-1)^{p(j)} : \varphi^{j} J^{([f, u_{j}])} : , \,\,I_{D} = \underset{j \in S_{1}}{\sum} \big(k (f|u_{j}) +
    \kappa_{>0}(f,u_j)\big)\partial \varphi^j.
\end{equation}
By (\ref{eq:2.13}) for $v = f$ and the last formula in (\ref{eq:2.6}), using that $d_{(0)}$ is an odd derivation of the vertex algebra $\C^{k}(\fg ,x, f)$, one obtains that element II is equal to the sum of the following five elements:
\begin{equation}
  \label{eq:4.9}
  II_{A} = \underset{j \in S_{1/2}}{\sum} (-1)^{p(j)}
    : \varphi^{j} J^{([f ,u_j])} :, \,\, II_{B} =
  \underset{j \in S_{1/2}}{\sum}\ \underset{i \in S_{3/2}}{\sum}
    (f | [[f ,u_{j}] , u_{i}]) : \Phi^{j} \varphi^{i} :,
\end{equation}
\begin{equation}
  \label{eq:4.10}
    II_{C} = \underset{j \in S_{1/2}}{\sum} \ \underset{i \in S_{1}}{\sum} (-1)^{p(j) (p(i) + 1)} : \Phi^{j} \varphi^{i} \Phi_{[[f, u_{j}] , u_{i}]} : , 
\end{equation}
\begin{equation}
  \label{eq:4.11}
    II_{D} = - \underset{i,j \in S_{1/2}}{\sum} (-1)^{p(i)(p(j)+1)} : \Phi^{j} \varphi^{i} J^{([[f,u_{j}],u_{i}])} : ,
\end{equation}
\begin{equation}
  \label{eq:4.12}
    II_{E} = \underset{i,j \in S_{1/2}}{\sum} \big(k ([f,u_{j}]|u_{i}) + \kappa_{> 0} ([f,u_{j}],u_{i})\big) : \Phi^{j} \partial \varphi^{i} :. 
\end{equation}
It is easy to see that $I_{A} + II_{B} = 0$ , and since also $I_{C} + II_{A} = 0$ , we obtain
\begin{equation}
  \label{eq:4.13}
    I + II = I_{B} + I_{D} + II_{C} + II_{D} + II_{E}.
\end{equation}

%\vspace{4mm}

\begin{lemma}
  \label{lem:4.3}. One has \\
  \\
  $(a) II_{C} = 0.$ \\
  \\
  $(b) II_{D} = \underset{i \in S_{1/2}}{\sum} \ \underset{k \in S_{0}}{\sum} : \Phi_{[u^{k} , u_{i}]} \varphi^{i} \ J^{(u_{k})} : .$
  \begin{proof}
    By (\ref{eq:4.10}), using that
\begin{displaymath}
  \Phi_{[[f,u_{j}],u_{k}]} = \underset{i \in S_{1/2}}{\sum} \langle u_{i} , [[f,u_{j}],u_{k}] \rangle^{\ne}  \Phi^{i} ,
\end{displaymath}
we obtain :
\begin{displaymath}
  II_{C} = \underset{i,j \in S_{1/2}}{\sum} \ \underset{k \in S_{1}}{\sum} (-1)^{p(j)(p(k) + 1)} \langle u_{i}, [[f,u_{j}],u_{k}] \rangle^{\ne}
  : \Phi^{j}\varphi^k\Phi^i: 
\end{displaymath}
\begin{displaymath}
  = \underset{i,j \in S_{1/2}}{\sum} \underset{k \in S_{1}}{\sum} (-1)^{(p(i) + p(j)(p(k)+1)} \langle u_{i}, [[f,u_{j}],u_{k}] \rangle^{\ne} : \Phi^{j} \Phi^{i} \varphi^k:
\end{displaymath}
\begin{displaymath}
  = \underset{i,j \in S_{1/2}}{\sum} \ \underset{k \in S_{1}}{\sum} \ ([[f,u_{i}] , [f,u_{j}]]  |  u_{k}) : \Phi^{j} \Phi^{i} \varphi^{k} : .
\end{displaymath}
If one exchanges $i$ and $j$ in the summation of the last expression, $II_{C}$ doesn't change.  On the other hand, looking at each summand in this expression, we see that it changes the sign, hence $II_{C} = - II_{C}$, proving (a).

By (\ref{eq:4.11}), using that, for $i , j \in S_{1/2}$, 
\begin{displaymath}
  [[f,u_{j}],u_{i}] = \underset{k \in S_{0}}{\sum} ([[f,u_{j}],u_{i}] | u^{k}) u_{k} ,
\end{displaymath}
we obtain :
\begin{displaymath}
  II_{D} = - \underset{i \in S_{1/2}}{\sum} \ \underset{k \in S_{0}}{\sum} (-1)^{p(i)p(k)} : \Phi_{[u_{i},u^{k}]} \varphi^{i} J^{(u_{k})} : ,
\end{displaymath}
proving (b).
\end{proof}
\end{lemma}

Next, we treat the term III.  For that introduce structure constants $c^{k}_{ij}$ and $c^{k}_{j} (v)$ for $i,j,k \in S_{> 0}$ and $v \in \fg_{0}$ by 
\begin{displaymath}
[u_{i},u_{j}] = \underset{k}{\sum} c^{k}_{ij} u_{k},\,\, [v,u_{j}] = \underset{k}{\sum} c^{k}_{j} (v) u_{k}.
\end{displaymath}
\begin{lemma}
  \label{lem:4.4} (a) For $v \in \fg_{0}$ and $k \in S_{>0}$ one has :
 \begin{equation}
\label{eq:4.14}
   [{\varphi^{k}}_{\lambda} J^{(v)}] = \underset{j \in S_{> 0}}{\sum} c^{k}_{j} (v) \varphi^{j} ,
\end{equation}
\begin{equation}
\label{eq:4.15}
  : \varphi^{k} J^{(v)} : - (-1)^{p(v)(p(k)+1)} : J^{(v)} \varphi^{k}:\, = \underset{j \in S_{> 0}}{\sum} c^{k}_{j} (v)\partial \varphi^j .
\end{equation}
(b) For $u \in \fg_0, v\in \fg_{1/2}$ and $j \in S_{> 0}$ one has
\begin{equation}
\label{eq:4.16}
  :: \Phi_{v} \varphi^{j} : J^{(u)} : = : \Phi_{v} : \varphi^{j} J^{(u)} :: + \underset{k \in S_{> 0}}{\sum} c^{j}_{k} (u) : (\partial \Phi_{v}) \varphi^k:,
\end{equation}
\begin{equation}
\label{eq:4.17}
  :: \varphi^{j} \Phi_{v} : J^{(u)} : = : \varphi^{j} : \Phi_{v} J^{(u)}: : + \underset{k \in S_{> 0}}{\sum} (-1)^{p(u)p(v)} c^{j}_{k} (u) : \varphi^{k} \partial \Phi_{v} : .
\end{equation}
\begin{proof}
  It uses the $\lambda$-bracket calculus, see \cite{K2}, \cite{DSK}.  Formula
  (\ref{eq:4.14}) follows by the non-commutative Wick formula, (\ref{eq:4.15}) by quasicommutativity, and (\ref{eq:4.16}) by quasiassociativity of a vertex algebra.  As an example, we prove here (\ref{eq:4.16}).  By quasiassociativity we have 
    \begin{displaymath}
      :: \Phi_{v} \varphi^{j} : J^{(u)} : - : \Phi_{v} : \varphi^{j} J^{(u)} : :
  \end{displaymath}
    \begin{displaymath}
      = (-1)^{p(u)(p(j)+1)} (\int^{- \partial}_{0} : \Phi_{v} [{J^{(u)}}_{\lambda} \varphi^{j}]  d \lambda : - : \Phi_{v} \int^{- \partial}_{0} [{J^{(u)}}_{\lambda} \varphi^{j}] d \lambda :). 
  \end{displaymath}
  Using (\ref{eq:4.14}), we obtain that the RHS is equal to 
\begin{displaymath}
  - \underset{k \in S_{> 0}}{\sum} c^{j}_{k} (u) (\int^{- \partial}_{0} : \Phi_{v} \varphi^{k}: d \lambda - : \Phi_{v} \int^{- \partial}_{0} \varphi^{k} d \lambda :) 
\end{displaymath}
\begin{displaymath}
  = \underset{k \in S_{> 0}}{\sum} c^{j}_{k} (u) (\partial : \Phi_{v} \varphi^{k} : - : \Phi_{v} \partial \varphi^{k} :) = \underset{k \in S_{> 0}}{\sum} c^{j}_{k} (u) : (\partial \Phi_{v}) \varphi^k:,
\end{displaymath}
proving (\ref{eq:4.16})
\end{proof}
\end{lemma}

We have, by formula (\ref{eq:2.13}), for $i \in S_{0}$ :
\begin{displaymath}
  d_{(0)} J^{(u_i)} \ = \underset{j \in S_{1}}{\sum} (f | [u_{i},u_{j}]) \varphi^{j} + \underset{j \in S_{1/2}}{\sum} (-1)^{p(i)(p(j)+1)} : \varphi^{j}
  \Phi_{[u_i, u_j]}:.
\end{displaymath}
It follows that
\begin{equation}
  \label{eq:4.18}
  d_{(0)} \underset{i \in S_{0}}{\sum} : J^{(u^{i})} \ J^{(u_{i})} : = A_{1} + A_{2} + A_{3} + A_{4} ,
\end{equation}
where
\begin{equation}
\label{eq:4.19}
  A_{1} = \underset{i \in S_{0}}{\sum} \ \underset{j \in S_{1}}{\sum} (f | [u^{i},u_{j}]) : \varphi^{j} J^{(u_{i})} : ,
\end{equation}
  \begin{equation}
\label{eq:4.20}
  A_{2} = \underset{i \in S_{0}}{\sum} \ \underset{j \in S_{1/2}}{\sum} (-1)^{p(i)(p(j)+1)} : : \varphi^{j} \Phi_{[u^{i},u_{j}]} : J^{(u_{i})} : ,
\end{equation}
\begin{equation}
    \label{eq:4.21}
    A_{3} = \underset{i \in S_{0}}{\sum} \ \underset{j \in S_{1}}{\sum} (-1)^{p(i)} (f | [u_{i},u_{j}]) : J^{(u^{i})} \varphi^{j} : ,
\end{equation}
\begin{equation}
    \label{eq:4.22}
    A_{4} = \underset{i \in S_{0}}{\sum} \ \underset{j \in S_{1/2}}{\sum} (-1)^{p(i)} (-1)^{p(i)(p(j)+1)} : J^{(u^{i})} \varphi^{j} \Phi_{[u_{i},u_{j}]} :.
\end{equation}

In order to simplify expressions for those elements, recall the operator $\Omega_{0}$, defined by $(\ref{eq:3.4})$. By Proposition 3.1, this operator is diagonalizable in $\fg_{j}$. Hence we can choose $u_{i} \in \fg_{1/2}$
(resp. $\fg_{1})$ to be eigenvectors of $\Omega_{0}$ ; denote by $a_{i}$ (resp. $b_{i}$) the corresponding eigenvalues.

We have by (4.16) :
\begin{displaymath}
  A_{2} = \underset{i \in S_{0}}{\sum} \ \underset{j \in S_{1/2}}{\sum} : \varphi^{j} \Phi_{[u^{i},u_{j}]} J^{(u_{i})} : + \underset{k \in S_{1/2}}{\sum} \ \underset{i \in S_{0}}{\sum} \ \underset{j \in S_{1/2}}{\sum}
   c^{j}_{i,k} :\varphi^k \partial \Phi_{[u^{i},u_{j}]}:.
\end{displaymath}
\noindent The first sum in this expression is equal to 
\begin{equation}
    \label{eq:4.23}
    A^{'}_{2} = \underset{i \in S_{0}}{\sum} \ \underset{j \in S_{1/2}}{\sum} : \Phi_{[u^{i},u_{j}]} \varphi^{j} J^{(u_{i})} : ,
\end{equation}
while the second sum is equal to
\begin{displaymath}
  \underset{k \in S_{1/2}}{\sum} : \varphi^{k} \underset{i \in S_{0}}{\sum} \ \
  \partial \Phi_{[u^{i},[u_{i},u_{k}]]} : .
\end{displaymath}
Hence we obtain
\begin{equation}
    \label{eq:4.24}
    A_{2} = A^{'}_{2} + \underset{k \in S_{1/2}}{\sum} a_{k} : \varphi^{k} \partial \Phi_{k}: .
\end{equation}
Next, we obtain, using
(\ref{eq:2.6}), (\ref{eq:4.14}) and (\ref{eq:4.21}),
\begin{displaymath}
  A_{3} = A_{1} - \underset{k \in S_{1}}{\sum}\underset{i \in S_{0}}{\sum}
  \underset{j \in S_{1}}{\sum} c_{i,k}^j(f|[u^i, u_j])\partial\varphi^k,
\end{displaymath}
hence
\begin{equation}
    \label{eq:4.25}
    A_{3} = A_{1} - \underset{k \in S_{1}}{\sum} b_{k} (f|u_{k}) \partial \varphi^{k},
\end{equation}
since
\begin{displaymath}
\underset{j \in S_{1}}{\sum} c_{i,k}^j(f|[u^i, u_j])=(f|\Omega_0(u_k)).
\end{displaymath}  
Finally, for $A_{4}$, given by (\ref{eq:4.22}), we have, using (\ref{eq:4.15}) and (\ref{eq:4.23}) :
\begin{equation}
    \label{eq:4.26}
    A_{4} = A^{'}_{2} - \underset{k \in S_{1/2}}{\sum} a_{k} : (\partial \varphi^{k}) \Phi_{k} : .
\end{equation}

From (\ref{eq:4.18}) - (\ref{eq:4.26})  we obtain that the element III is equal to the sum of the four elements 
\begin{equation}
    \label{eq:4.27}
    III_{A} = - \underset{i \in S_{0}}{\sum} \ \underset{j \in S_{1}}{\sum} (f | [u^{i},u_{j}]) : \varphi^{j} J^{(u_{i})} : ,
\end{equation}
\begin{equation}
    \label{eq:4.28}
    III_{B} = - \underset{i \in S_{0}}{\sum} \ \underset{j \in S_{1/2}}{\sum} : \Phi_{[u^{i},u_{j}]} \varphi^{j} J^{(u_{i})}:,
\end{equation}
\begin{equation}
    \label{eq:4.29}
    III_{C} = \frac{1}{2} \underset{j \in S_{1}}{\sum} b_{j} (f | u_{j}) \partial \varphi^{j} ,
\end{equation}
\begin{equation}
    \label{eq:4.30}
    III_{D} = \frac{1}{2} \underset{j \in S_{1/2}}{\sum} a_{j} \big(:(\partial \varphi^{j})\Phi_{j} : - : \varphi^{j} \partial \Phi_{j} : \big).
\end{equation}

We have :
\begin{displaymath}
  III_{A} = \underset{j \in S_{1}}{\sum} (-1)^{p(j)} : \varphi^{j} \underset{i \in S_{0}}{\sum} (u^{i} | [f,u_{j}]) J^{(u_{i})} : ,
\end{displaymath}
hence, by $(\ref{eq:4.7})$,
\begin{equation}
    \label{eq:4.31}
    III_{A} = - I_{B}.
\end{equation}
Using Lemma \ref{lem:4.3}(a), we obtain 
\begin{equation}
    \label{eq:4.32}
    II_{D} = - III_{B}.
\end{equation}
Hence, by $(\ref{eq:4.13})$ and Lemma \ref{lem:4.3}, we have
\begin{equation}
    \label{eq:4.33}
    I + II + III = I_{D} + II_{E} + III_{C} + III_{D}.
\end{equation}

Next, we have
\begin{equation}
    \label{eq:4.34}
    I_D + III_C = \underset{j \in S_1}{\sum} \big((k + h^\vee)(f|u_j)  + (\rho_{>0}|[f,u_j])\big) \partial \varphi^j,
\end{equation}
\begin{equation}
    \label{eq:4.35}
    II_E + III_D = (k + h^\vee) \underset{i \in S_{1/2}}{\sum} :  \Phi_i  \partial \varphi^i : \ - \frac{1}{2} \underset{i \in S_{1/2}}{\sum} :  (\partial \Phi_{\Omega_{0} u_{i}})  \varphi^i : \ + \underset{i \in S_{1/2}}{\sum} :  \Phi_{[u_{i}, \rho_{> 0}]}  \partial \varphi^i : .
\end{equation}
Indeed, by (\ref{eq:4.8}) and (\ref{eq:4.29}), we have, using (\ref{eq:4.3})
and (\ref{eq:2.8a}):
\begin{displaymath}
  I_D + III_C \ = \ \underset{j \in S_1}{\sum} \big((k + \frac{1}{2} b_j) (f|u_j) + h^\vee (f|u_j)  - \frac{1}{2} \kappa_0 (f,u_j)  + \frac{1}{2}
  \str_{\fg_{> 0}} \ad [f,u_j]\big) \partial \varphi^j.
\end{displaymath}
Applying (\ref{eq:4.4}) and (\ref{eq:4.6}) to the RHS, we obtain (\ref{eq:4.34}).

In order to prove (\ref{eq:4.35}), we rewrite (\ref{eq:4.30}) as follows :
\begin{displaymath}
  III_D \ = \ \frac{1}{2} \underset{k \in S_{1/2}}{\sum}
  \big(:  (\partial \varphi^k)  \Phi_{\Omega_0 u_k}  : -
  :  \varphi^k \partial  \Phi_{\Omega_{0} u_{k}}  :\big)  = \frac{1}{2} \underset{k \in S_{1/2}}{\sum}  \big(: \Phi_{\Omega_0 u_k} \partial \varphi^k:  -
  : (\partial \Phi_{\Omega_0 u_k})  \varphi^k  :\big).
\end{displaymath}
We also rewrite (\ref{eq:4.12}), using (\ref{eq:4.3}), (\ref{eq:4.4}) and (\ref{eq:4.6}), as follows :
\begin{displaymath}
II_E \ = \ (k + h^\vee) \ \underset{i \in S_{1/2}}{\sum} : \Phi_i  \partial \varphi^i  :  - \frac{1}{2} \underset{i \in S_{1/2}}{\sum}  :  \Phi_{\Omega_0 u_i}  \partial \varphi^i  :  +   \underset{i \in S_{1/2}}{\sum}  :  \Phi_{[u_i , \rho_{> 0}]} \partial \varphi^i : .
\end{displaymath}
\noindent Adding up these two expressions, we get (\ref{eq:4.35}).

\vspace{4mm}

\begin{lemma}
  \label{lem:4.5}
  One has
    \begin{displaymath}
  \Omega_0 (u) \ = \ 2 [\rho_{> 0} , u] \,\,\, \hbox{for}\,\,\, u \in \fg_{1/2}.
  \end{displaymath}
% In fact, property (\ref{eq:1.5}) suffices for this to hold. 
\end{lemma}
\begin{proof}
  Since elements I, II, III lie in the image of $d_{(0)}$ and $d^{2}_{(0)} = 0$ , we obtain, using (\ref{eq:4.33}) :
  \begin{displaymath}
    0 \ = \ d_{(0)} (I_D \ + \ III_C) \ + \ d_{(0)} (III_E \ + \ III_D).
  \end{displaymath}
  Substituting here (\ref{eq:4.34}) and (\ref{eq:4.35}) and using formulas (\ref{eq:2.6}) for the action of $d_{(0)}$ , we obtain :
  \begin{displaymath}
    0 \ = \underset{i, j \in S_{1/2}}{\sum}  (-1)^{p(i)}
    < u_j , \frac{1}{2} \Omega_0 (u_i)  - [\rho_{> 0} , u_i]>^{\ne} :
    \varphi^i \partial \varphi^j : .
  \end{displaymath}
  Due to the non-degeneracy of the bilinear form $<  .,.  >^{\ne}$, the lemma follows.
 % Note the we have used only the property (\ref{eq:1.5}) of the good datum.
%  grading (\ref{eq:1.1}).
\end{proof}

Finally, we treat the remaining three elements IV, V, and VI. Using (\ref{eq:2.13}), we obtain :
\begin{equation}
    \label{eq:4.36}
    IV  =  - (k + h^\vee) \underset{j \in S_1}{\sum} \ (f|u_j) \partial \varphi^j \ - \frac{k+h^\vee}{2}  \underset{j \in S_{1/2}}{\sum}  \partial : \Phi_j \varphi^j : ,
\end{equation}

\begin{equation}
    \label{eq:4.37}
    V  = - \underset{j \in S_1}{\sum}  (\rho_{> 0} | [f , u_j]) \partial \varphi^j  +  \underset{j \in S_{1/2}}{\sum} \partial  :\Phi_{[\rho_{> 0} , u_j]}
    \varphi^j:. 
\end{equation}
Using (\ref{eq:2.6}), we obtain 
\begin{equation}
    \label{eq:4.38}
    VI  =  \frac{k+h^\vee}{2}  \underset{j \in S_{1/2}}{\sum} ( : \varphi^j \partial \Phi_j :  -  : \Phi_j  \partial \varphi^j : ).
\end{equation}
Adding up (\ref{eq:4.36}) - (\ref{eq:4.38}), we obtain
\begin{equation}
    \label{eq:4.39}
    \begin{array}{l}
        IV + V + VI  = - (k + h^\vee) \underset{j \in S_1}{\sum} (f | u_j) \partial \varphi^j  - \underset{j \in S_1}{\sum} (\rho_{> 0} | [f , u_j]) \partial \varphi^j  \\
        - (k + h^\vee) \underset{j \in S_{1/2}}{\sum} : \Phi_j  \partial \varphi^j : + \underset{j \in S_{1/2}}{\sum} \partial  :  \Phi_{[\rho_{> 0} , u_j]} \varphi^j:.
    \end{array}
\end{equation}
Adding up (\ref{eq:4.33}) and (\ref{eq:4.39}), and using (\ref{eq:4.34}), (\ref{eq:4.35}) and Lemma \ref{lem:4.5}, we conclude that $d_{(0)} J^{\{ f \}} \ = \ 0$, completing the proof of Theorem \ref{th:3.1} .

\section{Proof of Theorem \ref{th:3.2}} 
\label{sec:5}

First, introduce the following convenient notation.  Let $\fa$ (resp. $\fa$') be the sum of some $\fg_{j}$'s (resp. the remaining $\fg_{j}$'s) in (\ref{eq:1.1}). Then we let $\delta_{u,\fa} = 1$ (resp. 0 ) if $u \in \fa$ (resp. $\fa$').
Then we have for $u$ , $v \in \fg$ :
\begin{equation}
\label{eq:5.1}
  \underset{i \in S_{> 0}}{\sum} ( u_{i} | v ) u^{i} \ = \ \delta_{v, \fg_{< 0}} v \text ; \,\, \underset{i \in S_{> 0}}{\sum} ( v | u^{i} ) u_{i} \ = \ \delta_{v, \fg_{> 0}}v; 
\end{equation}
\begin{equation}
  \label{eq:5.2}
    \underset{i \in S_{> 0}}{\sum} (u | u^{i}) (u_{i} | v) \ = \ \delta_{u , \fg_{> 0}} (u | v) \ = \ \delta_{v , \fg_{< 0}} (u | v); 
\end{equation}
\begin{equation}
  \label{eq:5.3}
  \underset{i \in S_{> 0}}{\sum} ( u | u^{i} ) ( v | u_{i} )
  \ = \ \delta_{u , \fg_{> 0}} ( v | u ) \ =
  \ \delta_{v , \fg_{< 0}} ( v | u ).
\end{equation}

\noindent Similar formulas hold if we replace $S_{> 0}$ by $S_{0}$ , and $\fg_{> 0}$ and $\fg_{< 0}$ by $\fg_{0}$ ; these formulas will be denoted by
(5.1)' , (5.2)' and (5.3)' .

Next, let $v^{\ch}$ denote the second summand on the right in (\ref{eq:2.7}). Then 
\begin{equation}
  \label{eq:5.4}
    v^{\ch} \ = \ \underset{i,j \in S_{> 0}}{\sum} (-1)^{p(i)} ([v , u_{j}] | u^{i} ) :\varphi_{i} \varphi^{j}: \ .
\end{equation}

Next, by condition (\ref{eq:1.5}), we have
\begin{equation}
  \label{eq:5.5}
    u^{i} \ = \ [ u^{(i)} , f ] ,\,\, i \in S_{1/2} \ ,
\end{equation}
where the $ \{ u^{(i)} \}_{i \in S_{1/2}}$ is a basis of $\fg_{1/2}$ , dual to $\{ u_{i} \}_{i \in S_{1/2}}$ with respect to the bilinear form
(\ref{eq:1.4}). 

Next, by the quasiassociativity of the normally ordered product, we have for $i ,j ,k ,l \in S_{> 0}$  
\begin{equation}
  \label{eq:5.6}
    : : \varphi_i \varphi^j : \varphi_k : \ = \ : \varphi_i \varphi^j \varphi_k : \ + \ (-1)^{p(j)} \delta_{j ,k} \partial \varphi_{i} \ ;
\end{equation}
\begin{equation}
\label{eq:5.7}  
    : : \varphi_i \varphi^j : \varphi_k : \ = \ : \varphi_i \varphi^j \varphi^k : \ + \ (-1)^{(p(i) +1)(p(j) + 1)} \delta_{i ,k} \partial \varphi^{j} \ ; 
\end{equation}
\begin{equation}
  \label{eq:5.8}
\begin{split}
  : : \varphi_i \varphi^j & : : \varphi_k \varphi^l : : \ = \ : \varphi_i \varphi^j \varphi_k \varphi^l : \ + \ (-1)^{p(k)} \delta_{j ,k}
  :( \partial \varphi_i) \varphi^l : \\
    & - \ (-1)^{p(j)p(k)} (-1)^{p(i)(p(j)+p(k))} \delta_{i ,l} : \varphi_{k} \partial \varphi^j : \ .
\end{split}
\end{equation}
\begin{lemma}
  \label{lem:5.1}
  We have, using (\ref{eq:5.4}) :
    \begin{equation}
\label{eq:5.9}
\underset{i \in S_0}{\sum} : u^i (u_i)^\ch : =
\underset{i, j \in S_{> 0}}{\sum} (-1)^{p(i)} : p_0 ([u_j , u^i])
\varphi_i \varphi^j : = \underset{i \in S_{0}}{\sum} : (u^{i})^\ch u_i :  ;
  \end{equation}
\begin{equation}
\label{eq:5.10}
\begin{split}
      \underset{i \in S_{0}}{\sum} : & (u^i)^\ch (u_i)^\ch : = \underset{i, j, k, l \in S_{> 0}}{\sum}(-1)^{p(i)+p(k)} ([u_l , u^k ] | p_0 [u_j , u^i]) : \varphi_i \varphi^j \varphi_k \varphi^l : \\
      & + \underset{\underset{[u_{j} , u^{i}] \in \fg_{0}}{i ,j ,k \in S_{> 0}}}{\sum} \ (-1)^{p(i)} (u_{j}  |  [ u^k , [ u_k , u^i ] ])
      ( :( \partial \varphi_i) \varphi^j : - : \varphi_i \partial \varphi^j :).
\end{split}
\end{equation}
    \begin{proof}
   Using the invariance of the bilinear form $(.  |  .)$ and (5.1)', we obtain 
      \begin{displaymath}
   \underset{i \in S_{0}}{\sum} : u^i ( u_i )^\ch : = \underset{j ,k \in S_{> 0}}{\sum} (-1)^{p(j)} : p_0 ([ u_k , u^j ]) \varphi_{j} \varphi^{k} : \ ,
   \end{displaymath}
      which is the first equality in (\ref{eq:5.9}) after replacing indices $j , k$ by $i , j$ . The proof of the second equality in (\ref{eq:5.9}) is the same. 
   
   Using the invariance of the bilinear form $(.  |  .)$ and (5.2)', we obtain 
      \begin{displaymath}
  % \begin{split}
        \underset{i \in S_{0}}{\sum}
       % &
        : (u^i)^\ch (u_i)^\ch : =
     \underset{j ,k ,r ,s \in S_{> 0}}{\sum} (-1)^{p(j) + p(r)}
    % \\   &
     ([u_s , u^r] | p_0 [u_k , u^j]): : \varphi_j  \varphi^k : : \varphi_{r} \varphi^{s}:: \ .
   %\end{split}
   \end{displaymath}
   Using (\ref{eq:5.8}), we see that this is equal to
      \begin{displaymath}
   = \underset{j, k, r, s \in S_{> 0}}{\sum} (-1)^{p(j) + p(r)} ([u_s , u^r ] \ | \ p_{0} [u_k , u^j]) : \varphi_j \varphi^k \varphi_{r} \varphi^{s} : 
   \end{displaymath}
      \begin{displaymath}
   + \underset{j, k, s \in S_{> 0}}{\sum} (-1)^{p(j)} ([u_s , u^k]  |  p_0 [u_k , u^j] ) : (\partial \varphi_j) \varphi^s :
   \end{displaymath}
      \begin{displaymath}
   - \underset{j, k, s \in S_{> 0}}{\sum} (-1)^{p(j)} (u_s  |  [u^k , p_0 [u_k , u^j] \ ] ) : \varphi_j \partial \varphi^s : \ .
   \end{displaymath}
   In the last term we used the invariance of $(.  |  .)$ and relabeling of indices; we also used that $(a|b) \neq 0  $ implies that $p(a) = p(b)$ in order to simplify the sign. Now (\ref{eq:5.10}) easily follows.
     \end{proof}
  \end{lemma}
%\vspace{4mm}
\begin{lemma}
  \label{lem:5.2}
  Recalling that $[u_i , u_j] = \underset{k}{\sum} c^{k}_{ij} u_{k}$ for $i, j, k \in S_{> 0}$ , we have
    \begin{equation}
      \label{eq:5.11}
      \underset{i, j, k, \in S_{> 0}}{\sum} (-1)^{p(i) + p(k)} c^{k}_{ij}
      :\varphi_k\varphi^ju^i: = \underset{i ,j \in S_{> 0}}{\sum} (-1)^{p(i)}
      : p_{< 0} ([u_j , u^i]) \varphi_i \varphi^j : \ ,
  \end{equation}
    \begin{equation}
      \label{eq:5.12}
      \underset{i, k \in S_{> 0}}{\sum} (-1)^{p(i)} : [u_k , u^i] \varphi_i \varphi^k : - \underset{i, j, k \in S_{> 0}}{\sum} (-1)^{p(i) + p(k)} c^{k}_{ij} : \varphi_k \varphi^j u^i : 
  \end{equation}
    \begin{displaymath}
  = \underset{i, j \in S_{> 0}}{\sum} (-1)^{p(i)} : p_{\geq 0} ([u_j , u^i]) \varphi_i \varphi^j : \ .
  \end{displaymath}
\begin{proof}
Using that $c^{k}_{ij} = ([u_i , u_j]  |  u^k)$, that the the bilinear form $(. \ | \ .)$ is invariant, equation (\ref{eq:5.1}), and that $p(i) + p(k) \ = \ p(j)$ if $c^{k}_{ij}\neq \ 0$, we obtain :
\begin{displaymath}
\underset{i,j,k \in S_{> 0}}{\sum} (-1)^{p(i) + p(k)} c^{k}_{ij} : \varphi_k \varphi^j u^i : = \underset{j,k \in S_{> 0}}{\sum} (-1)^{p(j)} : \varphi_k \varphi^j p_{< 0} [u_j , u^k] :  ,
\end{displaymath}
from which (\ref{eq:5.11}) follows.

By (\ref{eq:5.11}), the LHS of (\ref{eq:5.12}) is equal to 
\begin{displaymath}
(\underset{i,j \in S_{> 0}}{\sum} \ - \ \underset{\underset{[u_{j} , u^{i}] \in \fg_{< 0}}{i , j \in S_{> 0}}}{\sum}) (-1)^{p(i)} :[u_j,u^i] \varphi_i \varphi^j : .
\end{displaymath}
Formula (\ref{eq:5.12}) follows.
\end{proof}
\end{lemma}
\begin{lemma}
  \label{lem:5.3} The expression
  \begin{displaymath}
    A_{< 0} \ = \ \underset{i, j, k, l \in S_{> 0}}{\sum} (-1)^{p(i) + p(k)}
    ([u_l , u^k]  |  p_{< 0}[u_j , u^i]) : \varphi_i \varphi^j \varphi_k \varphi^l :
  \end{displaymath}
is equal to $\frac{1}{2} A_{\neq 0}$, where $A_{\neq 0}$ is obtained from $A_{< 0}$ by replacing $p_{< 0}$ by $p_{\neq 0}$ . 
\begin{proof}
 Exchanging $i$ with $k$ and $j$ with $l$ in $A_{< 0}$ , we obtain
 \begin{displaymath}
 A_{< 0} \ = \ \underset{i, j, k, l}{\sum} (-1)^{p(i) + p(k)} ([u_l , u^k] \ | \ p_{> 0}[u_j , u^i]) : \varphi_i \varphi^j \varphi_k \varphi^l : . 
 \end{displaymath}
 Adding the two expressions for $A_{< 0}$, we obtain $A_{\neq 0}$.
\end{proof}
\end{lemma}

From (\ref{eq:5.4}) we obtain
\begin{equation}
    \label{eq:5.13}
    f^{\ch} \ = \ \underset{i, j \in S_{> 0}}{\sum} (-1)^{p(i)}
    (f  | [u_j , u^i]) : \varphi_i \varphi^j : .
\end{equation}
\begin{lemma}
  \label{lem:5.4}
  We have 
  \begin{displaymath}
    \underset{i \in S_{1/2}}{\sum} (-1)^{p(i)} : \Phi^{i} [f , u_i]^{\ch} : = \underset{i, j \in S_{> 0}}{\sum} (-1)^{p(j)}
    : \Phi_{[u_j , u^i]} \varphi_i \varphi^j: .
  \end{displaymath}
  \begin{proof}
    Substituting in the LHS the expression (\ref{eq:5.4}) for $v = [f , u_i]$,
    we obtain, by invariance of $(.  |  .)$ and (\ref{eq:1.4}) , 
   \begin{displaymath}
   \underset{i \in S_{1/2}}{\sum} \underset{j, k \in S_{> 0}}{\sum} (-1)^{p(i) + p(k)} \ \langle u_i , [ u_j , u^k ] \rangle^{\ne} : \Phi^{i}\varphi_{k} \varphi^j :
   \end{displaymath}
   \begin{displaymath}
     = \underset{j, k, \in S_{> 0}}{\sum} (-1)^{p(j)}
     \underset{i \in S_{1/2}}{\sum} \langle u_i , [ u_j , u^k ] \rangle^{\ne} : \Phi^{i}  \varphi_{k} \varphi^{j} :
   \end{displaymath}
   \begin{displaymath}
   = \underset{j, k \in S_{> 0}}{\sum} (-1)^{p(j)} : \Phi_{[ u_j , u^k ]} \varphi_k \varphi^j : ,
   \end{displaymath}
   \noindent proving the lemma.
  \end{proof}
\end{lemma}
\begin{lemma}
  \label{lem:5.5}
  Let
  \begin{displaymath}
  P_0 \ = \ -f \ -\underset{i \in S_{1/2}}{\sum} (-1)^{p(i)} : \Phi^{i} [ f , u_i ] : + 1/2 \underset{i \in S_0}{\sum} : u^i u_i :
  \end{displaymath}
  \begin{displaymath}
  - \partial \rho_{> 0} + ( k + h^\vee ) \partial J^{(x)} - \frac{k + h^\vee}{2} \underset{i \in S_{1/2}}{\sum} : \Phi^{i} \partial \Phi_{i} :
  \end{displaymath}
  \begin{displaymath}
  - h^\vee \underset{i, \in S_{> 0}}{\sum} (-1)^{p(i)} : \varphi_i \partial \varphi^{i} : - \underset{i, j, k, \in S_{> 0}}{\sum} (-1)^{p(i) + p(k)} c^{k}_{ij} : \varphi_k \varphi^j u^i :
  \end{displaymath}
  \begin{displaymath} 
  + \underset{i, k \in S_{> 0}}{\sum} (-1)^{p(i)} : [ u_k , u^i ] \varphi_i \varphi^k : .
  \end{displaymath}
\noindent Then
\begin{equation}
    \label{eq:5.14}
    ( k + h^\vee ) L = d_{(0)} ( \underset{i \in S_{> 0}}{\sum} (-1)^{p(i)} : \varphi_i u^i :  ) + P_0 .
\end{equation}
\begin{proof}
 By (\ref{eq:3.2}) we have 
 \begin{equation}
     \label{eq:5.15}
     ( k + h^\vee ) L = \frac{1}{2} \underset{j \in S}{\sum} : u^j u_j : + ( k + h^\vee ) \partial x + ( k + h^\vee ) L^{\ch} + ( k + h^\vee ) L^{\ne} .
 \end{equation}
 \noindent Choosing, as usual, dual bases $ \{ h_i \} $ and $ \{ h^i \}$, $ i = 1 , \cdots , l $, of $\fh$ and root vectors $ \{ e_\alpha \}_{\alpha \in \Delta_{+}}, \{ e_{-\alpha} \}_{\alpha \in \Delta_{+}}$ of $\fg$, where
 $( e_\alpha  | e_{- \alpha}) = 1 $, we obtain, using quasicommutativity of
 the normally ordered product, that the first term in the RHS of (\ref{eq:5.15}) is 
\begin{equation}
\label{eq:5.16}
\underset{\alpha \in \Delta_{+}}{\sum} (-1)^{p(\alpha)} : e_{\alpha}
e_{ - \alpha} : + \frac{1}{2} \sum^{l}_{i = 1} h^i h_{i} - \partial \rho = \underset{i \in S_{> 0}}{\sum} (-1)^{p(i)} : u_i u^i : + \frac{1}{2} \underset{i \in S_{0}}{\sum} : u^i u_i : - \partial \rho_{> 0} .
\end{equation}
\noindent We also have 
\begin{equation}
    \label{eq:5.17}
    \partial x^{\ch} = \underset{i \in S_{}> 0}{\sum} (-1)^{p(i)} \ m_i \ \partial ( : \varphi_i \varphi^i :) \ .
\end{equation}
Using
(\ref{eq:5.16}) and (\ref{eq:5.17}), equation (\ref{eq:5.15})
can be rewritten as follows:
\begin{equation}
    \label{eq:5.18}
    \begin{split}
    & ( k + h^\vee ) L = \underset{i \in S_{> 0}}{\sum} (-1)^{p(i)} : u_i u^i : + \frac{1}{2} \underset{i \in S_{0}}{\sum} : u^i u_i : - \partial \rho_{> 0} \\
    & + ( k + h^\vee )\partial J^{(x)} - ( k + h^\vee ) \underset{i \in S_{> 0}}{\sum} (-1)^{p(i)} : \varphi_i \partial \varphi^i : + \frac{ k + h^{\vee} }{2} \underset{i \in S_{1/2}}{\sum} : (\partial \Phi^i) \Phi_i : \ .
    \end{split}
\end{equation}

Next, we compute $d_{(0)} (: \varphi_i u^i : )$, $ i \in S_{> 0}$, using (\ref{eq:2.6}) and that $d_{(0)}$ is an odd derivation of the normally ordered product:
\begin{displaymath}
d_{(0)} ( : \varphi_{i} u^{i} : ) = : u_{i} u^{i} : + \underset{j , k \in S_{> 0}}{\sum} (-1)^{p(k)} c^{k}_{i j} : \varphi_{k} \varphi^{j} u^{i} : + ( f  |  u_{i} ) u^i
\end{displaymath}
\begin{displaymath}
+ (-1)^{p(i)} : \Phi_{u_i} [u^{(i)} , f] : - k : \varphi_{i} \partial\varphi^{i} : - \underset{k \in S_{> 0}}{\sum} : [ u_k , u^i ] \varphi_i \varphi^k : \ .
\end{displaymath}
We have used for the 3-rd term in the RHS that $ (f  |  u_i ) = 0 $ if $p(i) \neq 0 $, and formula (\ref{eq:5.5}) for the 4-th term.
It follows that 
\begin{equation}
  \label{eq:5.19}
  \begin{split}
& \underset{i \in S_{> 0}}{\sum} (-1)^{p(i)} d_{(0)} ( : \varphi_{i} u^{i} : ) = \underset{i \in S_{> 0}}{\sum} (-1)^{p(i)} : u_i u^i : + \underset{i , j , k \in S_{> 0}}{\sum} (-1)^{p(i) + p(k)} c^{k}_{i , j} : \varphi_{k} \varphi^{j} u^{i} :\\
        & + f \ - \underset{i \in S_{1/2}}{\sum} (-1)^{p(i)} : \Phi^{i} [ u_i , f ] :  \\
        & - \underset{i , k \in S_{> 0}}{\sum} (-1)^{p(i)} : [ u_k , u^i ] \varphi_i \varphi^k : - k \underset{i \in S_{> 0}}{\sum} (-1)^{p(i)} : \varphi_i \partial \varphi^i : \ .
    \end{split}
\end{equation}
We have used for the 3-rd term in the RHS that $f = \underset{i \in S_{> 0}}{\sum} (-1)^{p(i)} ( f  |  u_{i}) u^i  $ , and for the 4-th term that 
\begin{displaymath}
\underset{i \in S_{1/2}}{\sum} : \Phi_{i} [ u^{(i)} , f  ] : = \underset{i \in S_{1/2}}{\sum} (-1)^{p(i)} d_{(0)} ( : \varphi_i u^i : ).
\end{displaymath}
Therefore $(k+h^\vee)L - \sum_{i\in S_{>0}}(-1)^{p(i)}d_{(0)}(:\varphi_iu^i:)$ 
is the difference of the right hand sides of equations (\ref{eq:5.18}) and (\ref{eq:5.19}), which is $P_{0}$ \ .
\end{proof}
\end{lemma}
\begin{lemma}
  \label{lem:5.6}
  We have
  \begin{displaymath}
  P_0 \ = - J^{ \{ f \} } + P_1 \ ,
  \end{displaymath}
  \noindent where
  \begin{displaymath}
  P_1 \ = \underset{i , j \in S_{> 0}}{\sum} (-1)^{p (i) } ( f  |  [u_j , u^i]) : \varphi_i \varphi^j : + \underset{i , j \in S_{> 0}}{\sum} (-1)^{p(j)} : \Phi_{[u_j , u^i]} \varphi_i \varphi^j : 
  \end{displaymath}
  \begin{displaymath}
  - \frac{1}{2} \underset{i ,j ,k. \ell \in S_{> 0}}{\sum} (-1)^{p(i) + p(k)} ([u_\ell , u^k]  |  p_0 [u_j , u^i]) : \varphi_i \varphi^j \varphi_k \varphi^{\ell}:
  \end{displaymath}
  \begin{displaymath}
  - \frac{1}{2} \underset{i ,j ,k \in S_{> 0}}{\sum} (-1)^{p(i)} (u_j | [u^k , p_0 [u_k , u^i]]) (: \partial \varphi_i \varphi^j : - : \varphi_i \partial \varphi^j :)
  \end{displaymath}
  \begin{displaymath}
  + \underset{i ,j \in S_{> 0}}{\sum} (-1)^{p(i)} (\rho_{> 0} | [u_j , u^i]) \partial : \varphi_i \varphi^j :
  \end{displaymath}
  \begin{displaymath}
    - h^{\vee} \underset{i \in S_{> 0}}{\sum} (-1)^{p(i)} \ : \varphi_i \partial \varphi^i \ + \underset{i ,j \in S_{> 0}}{\sum} (-1)^{p(i)}
    : p_{> 0} ([u_j , u^i]) \varphi_i \varphi^j :  .
  \end{displaymath}
Consequently, by (\ref{eq:5.14}), we have 
\begin{equation}
    \label{eq:5.20}
    (k + h^{\vee}) L + J^{ \{ f \}} \ \equiv \ P_1 \mod \Im d_{(0)}  .  
\end{equation}
\begin{proof}
 First, we compute, using Lemma \ref{lem:5.1},
 \begin{equation}
     \label{eq:5.21}
     \begin{split}
         & \underset{i \in S_{0}}{\sum} : J^{(u^i)} J^{(u_i)} : \ = \underset{i \in S_0}{\sum} : u^i u_i : +2 \underset{i , j \in S_{> 0}}{\sum} (-1)^{p(i)} : p_0 ([u_j , u^i]) \varphi_i \varphi^j : \\
         & + \underset{i,j,k, \ell \in S_{> 0}}{\sum} (-1)^{p(i) + p(k)} ([ u_{\ell} , u^k ]  |  p_0 [ u_j , u^i ]) : \varphi_i \varphi^j \varphi_k \varphi^{\ell} : \\
         & + \underset{i,j, k \in S_{> 0}}{\sum} (-1)^{p(i)} (u_j  | [ u^k , p_0 [u_k , u^i]]) (: \partial \varphi_i \varphi^j : - : \varphi_i \partial \varphi^j :)  .
     \end{split}
 \end{equation}
 Hence, for $P_0$, defined in Lemma (\ref{eq:5.5}), and $J^{ \{ f \}}$, defined in Theorem \ref{th:3.1}, we have 
  \begin{equation}
     \label{eq:5.22}
     P_0 + J^{ \{ f \} } \ = A  ,
 \end{equation}
  \noindent where
  \begin{equation}
     \label{eq:5.23}
     \begin{split}
         & A \ = (f^{\ch} + \underset{i \in S_{1/2}}{\sum} (-1)^{p(i)} : \Phi^i [f , u_i]^{\ch} : ) + \frac{1}{2} \underset{i \in S_{0}}{\sum} (: u^i u_i : - : J^{(u^{i})} J^{(u_i)} : ) \\
         & + \underset{i ,j \in S_{> 0}}{\sum} (-1)^{p(i)} (\rho_{> 0}| [u_j , u^i] ) \partial: \varphi_i \varphi^j: - h^{\vee} \underset{i \in S_{> 0}}{\sum} (-1)^{p(i)} : \varphi_i \partial \varphi^i : \\
         & + \underset{i , j \in S_{> 0}}{\sum} (-1)^{p(i)} : p_{\geq 0} ([u_j , u^i]) \varphi_i \varphi^j:  .
     \end{split}
 \end{equation}
  
  \noindent Here we used Lemma \ref{lem:5.4} for the first term, formula (\ref{eq:5.21}) for the second term and formula (\ref{eq:5.12}) for the last term.
 
 From (\ref{eq:5.23}) it is straightforward to deduce that $A = P_1$.
  This completes the proof of Lemma \ref{eq:5.6}.
 \end{proof}
\end{lemma}
\begin{lemma}
  \label{lem:5.7}
  Let
  \begin{displaymath}
    P_2 \ = - \frac{1}{2} \underset{i ,j ,k \in S_{> 0}}{\sum} (-1)^{p(i)}
    (u_j  |  [u^k , p_0 [u_k , u^i]] ) (: (\partial \varphi_i) \varphi^j : - : \varphi_i \partial \varphi^j :)
  \end{displaymath}
  \begin{displaymath}
  + \underset{i ,j \in S_{> 0}}{\sum} (-1)^{p(i)} (\rho_{> 0}  |  [u_j , u^i] ) \partial (: \varphi_i \varphi^j:) - h^{\vee} \underset{i \in S_{> 0}}{\sum} (-1)^{p(i)} : \varphi_i \partial \varphi^i :
  \end{displaymath}
  \begin{displaymath}
  - \frac{1}{2} \underset{i ,j ,k \in S_{> 0}}{\sum} (-1)^{p(i)} (u_j | [u^k , p_{< 0} [u_k , u^i]]) : (\partial \varphi_i) \varphi^j :
  \end{displaymath}
  \begin{displaymath}
  + \underset{i ,j ,k \in S_{> 0}}{\sum} (-1)^{p(i)} (u_j  |  [u^k , p_{> 0} [u_k , u^i]]) : \varphi_i \partial \varphi^j :  .
  \end{displaymath}
  Then
  \begin{displaymath}
  P_2 \ = P_1 - \frac{1}{2} d_{(0)} \underset{i ,j \in S_{> 0}}{\sum}(-1)^{p(j)} : \varphi_{[u_{i} , u^{j}]} \varphi_j \varphi^i :  .
  \end{displaymath}
  \noindent Consequently, by (\ref{eq:5.20}) , we have 
  \begin{displaymath}
  (k + h^{\vee}) L + J^{ \{ f \} } \ \equiv \ P_2 \mod \Im d_{(0)}  .
  \end{displaymath}
  \begin{proof}
   It is similar to that of Lemma \ref{lem:5.6}, and therefore is omitted  .
  \end{proof}
  \end{lemma}
  \begin{lemma}
    \label{lem:5.8}
    Let
    \begin{displaymath}
    \varphi = \underset{i ,j \in S_{>0} }{\sum} (a_{ij} : (\partial \varphi_i) \ \varphi^j : \ + b_{ij} \ : \varphi_i \ \partial \varphi^j :)  ,
    \end{displaymath}
    where $a_{ij} , b_{ij}  \in \mathbb{F}$.
    \noindent Then $d_{(0)} \varphi = 0$ implies that $\varphi = 0$.
    \begin{proof}
     It is clear from (\ref{eq:2.6}).
    \end{proof}
\end{lemma}
    % \begin{proof}
     Now it is easy to complete the proof of Theorem \ref{th:3.2}. By Lemma \ref{lem:5.7}, $(k + h^{\vee}) L + J^{ \{ f \} } \ \equiv \ P_2 \mod \Im d_{(0)}$ , where $d_{(0)}  P_2  = 0$ since $d_{(0)} L = 0 = d_{(0)} J^{ \{ f \} }$. But $P_2$ has the form of $\varphi$ in Lemma \ref{lem:5.8}, hence $P_2 \ = 0$.
%    \end{proof}
%  \end{lemma}
%\end{lemma}

%\vspace{20mm}

%.\\

\section{ { Examples }}
\label{sec:6}

\noindent \textbf{6.1 Minimal W-algebras.} Let $\theta\in \fh^* =\fh$ be the highest root for some ordering of roots of the Lie superalgebra $\fg$.
The W-algebra $W^{k} (\fg, \theta)$ is called a {\it minimal} W-algebra \cite{KRW}, \cite{KW} if the  $\frac{1}{2} \mathbb{Z}$-grading (\ref{eq:1.1}) has the form
\begin{displaymath}
  \fg  =  \oplus_{j = -1}^{1}  \fg_{j} ,\, \text{where} \ \fg_{-1}  =  \mathbb{F}  e_{- \theta} .
\end{displaymath}
In this case $f = e_{- \theta}$ lies in the non-zero nilpotent orbit of minimal dimension in one of the simple components of
$\fg_{\overline{0}}$. 
Conversely, if $f$ lies in the non-zero orbit of minimal dimension in a simple component of $\fg_{\overline{0}}$, then the corresponding $W$-algebra is a minimal $W$-algebra in all cases, except when $\fg=osp(3|n)$ and the simple component of  $\fg_{\overline{0}}$ is $so_3$. 
Minimal $W$-algebras were studied in detail in \cite{KRW} and \cite{KW}.

Obviously, for a minimal $W$-algebra, $\rho_{1} = x$, and it follows from \cite{KW}, formulas (5.6), (5.11), that $\rho_{1/2} = (h^{\vee} - 2)x$.  Hence,
\begin{equation}
    \label{eq:6.1}
    \rho_{> 0} \ = \ (h^{\vee} - 1)x.
\end{equation}
Therefore, $\rho_{> 0} - (k + h^{\vee})x \ = \ -(k + 1)x$, and the FFR, given by Corollary \ref{cor:3.1}, coincides with that, given by \cite{KW}, Theorem 5.3.
\vspace{3mm}

\noindent \textbf{6.2 Principal W-algebras.} Let $\{e_*, \rho^\vee, f_*\}$ be
  a principal $sl_2$-triple, where
  %in $\fg_{\overline{0}}$ (it is called a principal nilpotent of
%$\fg_{\overline{0}}$),
%Then the only good grading (\ref{eq:1.1}) with $f \in \fg_{-1}$ is the Dynkin grading,
$x = \rho^{\vee}$ is the half of the sum of positive coroots of
$\fg_{\overline{0}}$. Then the datum $(\fg, \rho^\vee, f_*, k)$ is a Dynkin datum.
%of the even part of $\fg$.
%, where $\rho^{\vee} \in \fh$ is defined by $[\rho^{\vee}, e_{i}] \ = \ e_{i}$ for all simple root vectors $e_{i}$.
The corresponding W-algebra $W^{k} (\fg, \rho^{\vee}, f_*)$ is called the \textit{principal} W-algebra, associated to $\fg$.

If $\fg$ is a Lie algebra, then $\fg_{\pm 1/2} = 0$ and $\fg_0 = \fh$, and therefore
\begin{equation}
  \label{eq:6.2}
  \rho_{> 0}  =  \rho\, (\in \fh^{*}  =  \fh) ,\,\,\hbox{and}\,\, \fg^{f_*}_0=0,
\end{equation}
where $\rho$ is the half of the sum of positive roots of $\fg$.
%\ = \ \underset{j \geq 0}{\sum} \ \rho_{j}$. 
Hence the FFR in this case is a homomorphism
$W^{k} (\fg, \rho^{\vee}, f_*) \to V^{B_{0}} (\fh)$,
for which 
\begin{equation}
  \label{eq:6.3}
L \ \mapsto \ \frac{1}{2(k + h^{\vee})}  \underset{j \in S_{0}}{\sum}  :  u^{j} u_{j}  :  +  \partial \rho^\vee - \frac{1}{k+h^\vee}\partial \rho.
\end{equation}

The principal $W$-algebras for arbitrary simple Lie algebras were first constructed in \cite{FF}.

The element $2x$ is determined by its Dynkin labels
$2\alpha_i(x),\, i=1,...,\rank \,\fg$, which are known to take values
$0, 1$, and $2$. In the case when $\fg$ is  a simple Lie algebra
all the Dynkin labels of $2\rho^\vee$ are equal to 2.

Let now $\fg$ be a basic Lie superalgebra, which is not a Lie algebra.
Then $\fg$ may have several non-isomorphic sets of simple roots, and the
Dynkin diagrams, corresponding to the choices of positive roots, compatible
with the grading (\ref{eq:1.1}) may be different. Below we list the Dynkin labels $2\alpha_i(\rho^\vee),\, i=1,..., \rank\, \fg$, for all basic Lie superalgebras
$\fg$, which are not Lie algebras.
For exceptional Lie superalgebras $\fg$ they can be found in \cite{H2}. We use notation for basic Lie superalgebras and their Dynkin diagrams from \cite{K1}. 

\vspace{8mm}

\noindent $I$. $A(m, n), \,m > n \geq 0$, $m - n = 2k + 1$, $k \in
\mathbb{Z}_ {\geq 0}$ :

\begin{table*}[!h]
\vspace*{-15ex}$
\begin{array}{c c}
\setlength{\unitlength}{0.16in}
\begin{picture}(20,8)

\put(5,2){\makebox(0,0)[c]{2}}
\put(5,1){\makebox(0,0)[c]{$\bigcirc$}}

\put(5.5,1){\line(1,0){1.0}}

\put(7,2){\makebox(0,0)[c]{2}}
\put(7.05,1){\makebox(0,0)[c]{$\bigcirc$}}

\put(7.6,1){\line(1,0){1.0}}

\put(9.4,1){\makebox(0,0)[c]{$\cdots$}}

\put(10,1){\line(1,0){1.0}}

\put(11.52,2){\makebox(0,0)[c]{2}}
\put(11.52,1){\makebox(0,0)[c]{$\bigcirc$}}

\put(12.1,1){\line(1,0){1.0}}

\put(13.75,2){\makebox(0,0)[c]{1}}
\put(13.75,1){\makebox(0,0)[c]{$\bigotimes$}}

\put(14.3,1){\line(1,0){1.0}}

\put(15.95,2){\makebox(0,0)[c]{1}}
\put(15.95,1){\makebox(0,0)[c]{$\bigotimes$}}

\put(16.5,1){\line(1,0){1.0}}

\put(18.25,1){\makebox(0,0)[c]{$\cdots$}}

\put(18.9,1){\line(1,0){1.2}}

\put(20.75,2){\makebox(0,0)[c]{1}}
\put(20.75,1){\makebox(0,0)[c]{$\bigotimes$}}

\put(21.35,1){\line(1,0){1.0}}

\put(23,2){\makebox(0,0)[c]{2}}
\put(23,1){\makebox(0,0)[c]{$\bigcirc$}}

\put(23.5,1){\line(1,0){1.0}}

\put(25,2){\makebox(0,0)[c]{2}}
\put(25,1){\makebox(0,0)[c]{$\bigcirc$}}

\put(25.5,1){\line(1,0){1.0}}

\put(27.25,1){\makebox(0,0)[c]{$\cdots$}}

\put(28,1){\line(1,0){1.0}}

\put(29.5,2){\makebox(0,0)[c]{2}}
\put(29.5,1){\makebox(0,0)[c]{$\bigcirc$}}

\end{picture}
\end{array}$
\end{table*}
\noindent where the number of white nodes at the beginning and the end is equal to $k$,
and the number of grey nodes is $2(n + 1)$.

\vspace{5mm}

\noindent $II$. $A(m, n),\, m > n \geq 0$, $m - n = 2k$, $k \in \mathbb{Z}_{\geq 1}$ :

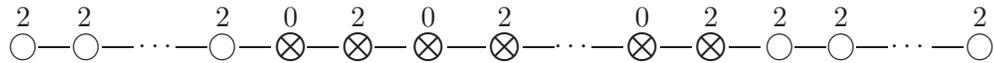
\begin{table*}[!h]
\vspace*{-15ex}$
\begin{array}{c c}
\setlength{\unitlength}{0.16in}
\begin{picture}(20,8)

\put(5,2){\makebox(0,0)[c]{2}}
\put(5,1){\makebox(0,0)[c]{$\bigcirc$}}

\put(5.5,1){\line(1,0){1.0}}

\put(7,2){\makebox(0,0)[c]{2}}
\put(7.05,1){\makebox(0,0)[c]{$\bigcirc$}}

\put(7.6,1){\line(1,0){1.0}}

\put(9.4,1){\makebox(0,0)[c]{$\cdots$}}

\put(10,1){\line(1,0){1.0}}

\put(11.52,2){\makebox(0,0)[c]{2}}
\put(11.52,1){\makebox(0,0)[c]{$\bigcirc$}}

\put(12.1,1){\line(1,0){1.0}}

\put(13.75,2){\makebox(0,0)[c]{0}}
\put(13.75,1){\makebox(0,0)[c]{$\bigotimes$}}

\put(14.3,1){\line(1,0){1.0}}

\put(15.95,2){\makebox(0,0)[c]{2}}
\put(15.95,1){\makebox(0,0)[c]{$\bigotimes$}}

\put(16.5,1){\line(1,0){1.0}}

\put(18.25,2){\makebox(0,0)[c]{0}}
\put(18.25,1){\makebox(0,0)[c]{$\bigotimes$}}

\put(18.9,1){\line(1,0){1.2}}

\put(20.75,2){\makebox(0,0)[c]{2}}
\put(20.75,1){\makebox(0,0)[c]{$\bigotimes$}}

\put(21.35,1){\line(1,0){1.0}}

\put(23,1){\makebox(0,0)[c]{$\cdots$}}

\put(23.6,1){\line(1,0){1.0}}

\put(25.25,2){\makebox(0,0)[c]{0}}
\put(25.25,1){\makebox(0,0)[c]{$\bigotimes$}}

\put(25.85,1){\line(1,0){1.0}}

\put(27.5,2){\makebox(0,0)[c]{2}}
\put(27.5,1){\makebox(0,0)[c]{$\bigotimes$}}

\put(28.1,1){\line(1,0){1.0}}

\put(29.75,2){\makebox(0,0)[c]{2}}
\put(29.75,1){\makebox(0,0)[c]{$\bigcirc$}}

\put(30.25,1){\line(1,0){1.0}}

\put(31.75,2){\makebox(0,0)[c]{2}}
\put(31.75,1){\makebox(0,0)[c]{$\bigcirc$}}

\put(32.25,1){\line(1,0){1.0}}

\put(34,1){\makebox(0,0)[c]{$\cdots$}}

\put(34.75,1){\line(1,0){1.0}}

\put(36.3,2){\makebox(0,0)[c]{2}}
\put(36.3,1){\makebox(0,0)[c]{$\bigcirc$}}

\end{picture}
\end{array}$
\end{table*}
\noindent where the number of white nodes at the beginning (resp. end) is equal to $k$ (resp. $k - 1$), and the number of grey nodes is $2 (n + 1)$ .

\vspace{5mm}

\noindent $III$. $A (m, m),\, m \geq 1$ :

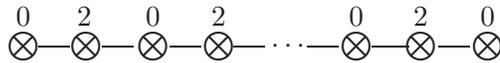
\begin{table*}[!h]
\vspace*{-15ex}$
\begin{array}{c c}
\setlength{\unitlength}{0.16in}
\begin{picture}(20,8)

\put(5,2){\makebox(0,0)[c]{0}}
\put(5,1){\makebox(0,0)[c]{$\bigotimes$}}

\put(5.5,1){\line(1,0){1.0}}

\put(7,2){\makebox(0,0)[c]{2}}
\put(7.05,1){\makebox(0,0)[c]{$\bigotimes$}}

\put(7.6,1){\line(1,0){1.0}}

\put(9.25,2){\makebox(0,0)[c]{0}}
\put(9.25,1){\makebox(0,0)[c]{$\bigotimes$}}

\put(9.8,1){\line(1,0){1.0}}

\put(11.4,2){\makebox(0,0)[c]{2}}
\put(11.4,1){\makebox(0,0)[c]{$\bigotimes$}}

\put(11.9,1){\line(1,0){1.0}}

\put(13.75,1){\makebox(0,0)[c]{$\cdots$}}

\put(14.3,1){\line(1,0){1.0}}

\put(15.9,2){\makebox(0,0)[c]{0}}
\put(15.9,1){\makebox(0,0)[c]{$\bigotimes$}}

\put(16.4,1){\line(1,0){1.2}}

\put(18,2){\makebox(0,0)[c]{2}}
\put(18,1){\makebox(0,0)[c]{$\bigotimes$}}

\put(18.6,1){\line(1,0){1.0}}

\put(20.2,2){\makebox(0,0)[c]{0}}
\put(20.2,1){\makebox(0,0)[c]{$\bigotimes$}}

\end{picture}
\end{array}$
\end{table*}
\noindent
In all cases I - III the total number of nodes is $m + n + 1$ .

%\vspace{5mm}

\newpage

\noindent $IV$. $B (m, n),\, m \geq 0, \, n \geq 1$ :

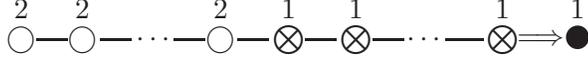
\begin{table*}[!h]
\vspace*{-15ex}$
\begin{array}{c c}
\setlength{\unitlength}{0.16in}
\begin{picture}(20,8)

\put(5,2){\makebox(0,0)[c]{2}}
\put(5,1){\makebox(0,0)[c]{$\bigcirc$}}

\put(5.5,1){\line(1,0){1.0}}

\put(7,2){\makebox(0,0)[c]{2}}
\put(7,1){\makebox(0,0)[c]{$\bigcirc$}}

\put(7.6,1){\line(1,0){1.0}}

\put(9.4,1){\makebox(0,0)[c]{$\cdots$}}

\put(10,1){\line(1,0){1.0}}

\put(11.52,2){\makebox(0,0)[c]{2}}
\put(11.52,1){\makebox(0,0)[c]{$\bigcirc$}}

\put(12.1,1){\line(1,0){1.0}}

\put(13.75,2){\makebox(0,0)[c]{1}}
\put(13.75,1){\makebox(0,0)[c]{$\bigotimes$}}

\put(14.3,1){\line(1,0){1.0}}

\put(15.95,2){\makebox(0,0)[c]{1}}
\put(15.95,1){\makebox(0,0)[c]{$\bigotimes$}}

\put(16.5,1){\line(1,0){1.0}}

\put(18.25,1){\makebox(0,0)[c]{$\cdots$}}

\put(18.9,1){\line(1,0){1.2}}

\put(20.75,2){\makebox(0,0)[c]{1}}
\put(20.75,1){\makebox(0,0)[c]{$\bigotimes$}}

\put(22,1){\makebox(0,0)[c]{$\Longrightarrow$}}

\put(23.2,2){\makebox(0,0)[c]{1}}
\put(23.2,1){\makebox(0,0)[c]{${\raisebox{-.75ex}{\Huge{\textbullet}}}$}}

\end{picture}
\end{array}$
\end{table*}
\noindent where the number of white nodes is $m - n$ if $m \geq n$, and is $n - m - 1$ if $m \geq n - 1$; the total number of nodes is $m + n$.

\vspace{5mm}

\noindent $V$. $C(n) \, , \, n \, \geq \, 3 $ :

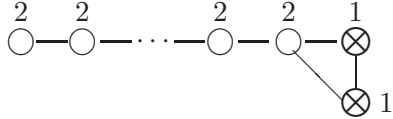
\begin{table*}[!h]
\vspace*{-15ex}$
\begin{array}{c c}
\setlength{\unitlength}{0.16in}
\begin{picture}(20,8)

\put(5,2){\makebox(0,0)[c]{2}}
\put(5,1){\makebox(0,0)[c]{$\bigcirc$}}

\put(5.5,1){\line(1,0){1.0}}

\put(7,2){\makebox(0,0)[c]{2}}
\put(7,1){\makebox(0,0)[c]{$\bigcirc$}}

\put(7.6,1){\line(1,0){1.0}}

\put(9.4,1){\makebox(0,0)[c]{$\cdots$}}

\put(10,1){\line(1,0){1.0}}

\put(11.52,2){\makebox(0,0)[c]{2}}
\put(11.52,1){\makebox(0,0)[c]{$\bigcirc$}}

\put(12.1,1){\line(1,0){1.0}}

\put(13.75,2){\makebox(0,0)[c]{2}}
\put(13.75,1){\makebox(0,0)[c]{$\bigcirc$}}

\put(14.3,1){\line(1,0){1.0}}

\put(15.95,2){\makebox(0,0)[c]{1}}
\put(15.95,1){\makebox(0,0)[c]{$\bigotimes$}}

\put(16.95,-1){\makebox(0,0)[c]{1}}
\put(15.95,-1){\makebox(0,0)[c]{$\bigotimes$}}

\put(13.9,0.7){\line(1,-1){1.6}}

\put(15.95,0.5){\line(0,-1){1}}

\end{picture}
\end{array}$
\end{table*}

\vspace{2mm}

\noindent where the number of white nodes is $n - 2$.

\vspace{5mm}

\noindent $VI$. $D (m, n),\,  \, m \geq 2 , \, n \geq 1$ :

\vspace{5mm}

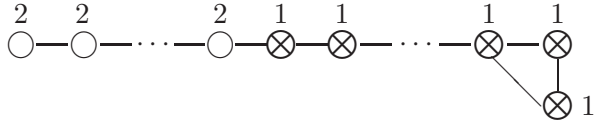
\begin{table*}[!h]
\vspace*{-17ex}$
\begin{array}{c c}
\setlength{\unitlength}{0.16in}
\begin{picture}(20,8)

\put(5,2){\makebox(0,0)[c]{2}}
\put(5,1){\makebox(0,0)[c]{$\bigcirc$}}

\put(5.5,1){\line(1,0){1.0}}

\put(7,2){\makebox(0,0)[c]{2}}
\put(7.05,1){\makebox(0,0)[c]{$\bigcirc$}}

\put(7.6,1){\line(1,0){1.0}}

\put(9.4,1){\makebox(0,0)[c]{$\cdots$}}

\put(10,1){\line(1,0){1.0}}

\put(11.52,2){\makebox(0,0)[c]{2}}
\put(11.52,1){\makebox(0,0)[c]{$\bigcirc$}}

\put(12,1){\line(1,0){1.0}}

\put(13.5,2){\makebox(0,0)[c]{1}}
\put(13.5,1){\makebox(0,0)[c]{$\bigotimes$}}

\put(14,1){\line(1,0){1.0}}

\put(15.5,2){\makebox(0,0)[c]{1}}
\put(15.5,1){\makebox(0,0)[c]{$\bigotimes$}}

\put(16.1,1){\line(1,0){1.0}}

\put(18,1){\makebox(0,0)[c]{$\cdots$}}

\put(18.7,1){\line(1,0){1.0}}

\put(20.3,2){\makebox(0,0)[c]{1}}
\put(20.3,1){\makebox(0,0)[c]{$\bigotimes$}}

\put(20.9,1){\line(1,0){1.0}}

\put(22.55,2){\makebox(0,0)[c]{1}}
\put(22.55,1){\makebox(0,0)[c]{$\bigotimes$}}

\put(23.55,-1){\makebox(0,0)[c]{1}}
\put(22.55,-1){\makebox(0,0)[c]{$\bigotimes$}}

\put(20.4,0.6){\line(1,-1){1.6}}

\put(22.55,0.5){\line(0,-1){1}}

\end{picture}
\end{array}$
\end{table*}

\vspace{2mm}

\noindent where the number of white nodes is $m - n - 1$, if $m \geq n +1$, and is $n - m$ if $m \leq n$; the total number of nodes is $m + n$ 

\vspace{5mm}

\noindent $VII$. $D(2, 1; a)$ :

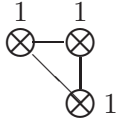
\begin{table*}[!h]
\vspace*{-14ex}
$
\begin{array}{c c}
\setlength{\unitlength}{0.16in}
\begin{picture}(20,8)

\put(6,2){\makebox(0,0)[c]{1}}
\put(6,1){\makebox(0,0)[c]{$\bigotimes$}}

\put(6.4,1){\line(1,0){1.0}}

\put(7.95,2){\makebox(0,0)[c]{1}}
\put(7.95,1){\makebox(0,0)[c]{$\bigotimes$}}

\put(8.95,-1){\makebox(0,0)[c]{1}}
\put(7.95,-1){\makebox(0,0)[c]{$\bigotimes$}}

\put(6.4,0.6){\line(1,-1){1.6}}

\put(7.95,0.5){\line(0,-1){1}}

\end{picture}
\end{array}$
\end{table*}

\vspace{2mm}

\noindent where the simple roots are $\{ \epsilon_{1} + \epsilon_{2} + \epsilon_{3} \, , \, \epsilon_{1} - \epsilon_{2} - \epsilon_{3} \, , \, - \epsilon_{1} - \epsilon_{2} + \epsilon_{3} \}$.

\vspace{5mm}

\noindent $VIII$. $F(4)$ :

\vspace{5mm}

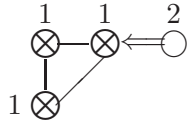
\begin{table*}[!h]
\vspace*{-17ex}$
%\begin{array}{c c}
%  
\setlength{\unitlength}{0.16in}
\begin{picture}(20,8)

\put(6,2){\makebox(0,0)[c]{1}}
\put(6,1){\makebox(0,0)[c]{$\bigotimes$}}

\put(6.4,1){\line(1,0){1.0}}

\put(7.95,2){\makebox(0,0)[c]{1}}
\put(7.95,1){\makebox(0,0)[c]{$\bigotimes$}}

\put(9.2,1){\makebox(0,0)[c]{$\Longleftarrow$}}

\put(10.2,2){\makebox(0,0)[c]{2}}
\put(10.2,1){\makebox(0,0)[c]{$\bigcirc$}}

\put(5,-1){\makebox(0,0)[c]{1}}
\put(6,-1){\makebox(0,0)[c]{$\bigotimes$}}

\put(7.95,0.6){\line(-1,-1){1.6}}

\put(6,0.5){\line(0,-1){1}}

\end{picture}
%
%\end{array}
$
\end{table*}

\vspace{5mm}
\noindent where the simple roots are $\{ \frac{1}{2} (\delta + \epsilon_{1} - \epsilon_{2} - \epsilon_{3}), \frac{1}{2} (\delta - \epsilon_{1} + \epsilon_{2} + \epsilon_{3}), \, \frac{1}{2} (- \delta + \epsilon_{1} - \epsilon_{2} + \epsilon_{3}), \, \epsilon_{2} - \epsilon_{3} \}$.

\vspace{5mm}

\noindent $IX$. $G(3)$ :

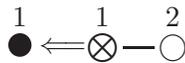
\begin{table*}[!h]
\vspace*{-14ex}
$
%\begin{array}{c c}
%  
\setlength{\unitlength}{0.16in}
\begin{picture}(20,8)

\put(5,2){\makebox(0,0)[c]{1}}
\put(5,1){\makebox(0,0)[c]{${\raisebox{-.75ex}{\Huge{\textbullet}}}$}}

\put(6.4,1){\makebox(0,0)[c]{$\Longleftarrow$}}

\put(7.7,2){\makebox(0,0)[c]{1}}
\put(7.7,1){\makebox(0,0)[c]{$\bigotimes$}}

\put(8.4,1){\line(1,0){1.0}}

\put(10,2){\makebox(0,0)[c]{2}}
\put(10,1){\makebox(0,0)[c]{$\bigcirc$}}

\end{picture}
%
%\end{array}
$
\end{table*}

\vspace{4mm}
\noindent where the simple roots are $\{ \delta, - \delta + \epsilon_{1}, \, \epsilon_{2} - \epsilon_1 \}$.

\vspace{4mm}
Looking at these diagrams, we see that for all basic Lie superalgebras $\fg$,
except for $A(m,n)$ with $m-n$ even, which we shall exclude from consideration,
the $\frac{1}{2}\ZZ$-grading (\ref{eq:1.1}), corresponding to the principal
nilpotent element, is defined by
\begin{displaymath}
  \alpha_i(\rho^\vee)=1\,\, (\hbox{resp}=\frac{1}{2})\,\,\hbox{if}\,\,
  \alpha_i\, \,\hbox{is even (resp. odd)},
\end{displaymath}
where $\alpha_i$, $i=1,...,\rank\,\fg$, are simple roots..
Hence this grading is compatible with the parity and $\fg_0=\fh$.
It follows that (\ref{eq:6.2}) still holds. Furthermore, $\fg_{1/2}$
(resp. $\fg_{-1/2}$) is a purely odd space, spanned by the
$e_{\alpha_i}$
(resp. $e_{-\alpha_i}$),
where the $\alpha_i$ are all odd simple roots.
The element $f\in \fg_{-1}$ can be chosen as follows. Let $f^0$ (resp. $f^1$)
be the sum of all $e_{-\alpha_i}$ with $\alpha_i$
even (resp. odd); then 
\begin{displaymath}
f=f^0 + [f^1,f^1].
\end{displaymath}

\begin{remark}
  \label{rem:6.1}
It is probably impossible to write down a complete FFR of an arbitrary $W$-algebra, with explicit expressions for all elements beyond those of conformal weight $1$, $\frac{3}{2}$,
and $L$. However, in many cases (including the minimal one) the $W$-algebra
$W^k(\fg, x,f)$
is generated by elements of conformal weight $1$ and $\frac{3}{2}$ (this happens, for example, when
$\fg^f_{-1/2}$
generates $\fg^f_{<0}$). In such cases formulas
(\ref{eq:2.14}) and (\ref{eq:2.17}) can be extended to the complete FFR of this $W$-algebra.
\end{remark}

\section{ Appendix A. A more conceptual proof of Theorem 3.2 for Dynkin datum}
\label{sec:7}

We give a proof under the following assumptions :

(i) the restriction of the bilinear form $(. \ | \ .)$ to $\fg^{f}_{0}$ is non-degenerate;
%and the associated Casimir operator is diagonalizable on $\fg_0^f$;
%a direct sum of its center and basic Lie superalgebras ;

(ii) (even part of $\fg_{-1}$) $\cap$ (center of $\fg^{f}$) $= \mathbb{F} f$.

Property (i)  holds for any basic Lie superalgebra $\fg$
%with a non-degenerate Killing form
and its Dynkin
grading (\ref{eq:1.1}).
%, by the representation theory of $sl_2$.
Property (ii) holds for a Dynkin grading by the Brylinski-Kostant Theorem \cite{BK},
\cite{P} in the Lie algebra case, and its analogue in the Lie superalgebra
case, which follows from \cite{H1} and \cite{H2}.
%These assumptions also hold for a Dynkin grading of a basic Lie superalgebra, provided that $f$ lies in a simple component of $\fg_{\overline{0}}$ .

\vspace{4mm}

\begin{lemma}
  \label{lem:7.1}
  Let L be the element of $\C^k (\fg , x, f)$, given by (\ref{eq:3.1}), and let $L^{'} = - \frac{1}{k + h^{\vee}} \ J^{\{ f \}}$ , where $J^{\{ f \}}$ is an element of $\C^k (\fg , x, f)$ given by Theorem \ref{th:3.1}. Let $a \in \fg^{f}_{0}$ and $J^{\{ a \}} \in \C^k (\fg , x, f)$ be defined by (\ref{eq:2.14}). Then \\

  (a) $[L_\lambda J^{\{ a \}}]  =  (\partial + \lambda) J^{\{ a \}}  +  \lambda^2 (\rho_{> 0} | a)$.
  
  (b) $[L^{'}_{\lambda} J^{\{ a \}}]  =  (\partial + \lambda) J^{\{ a \}}  +  \lambda^{2} (\rho_{> 0} | a)$. \\
\begin{proof}
From (\ref{eq:2.10}) and (\ref{eq:2.14}) we deduce 
\begin{equation}
    \label{eq:7.1}
    [{J^{\{ a \}}}_{\lambda} J^{(v)}]  =  J^{([a,v])} + \lambda B_{0} (a,b),\, a \in \fg^{f}_{0},\, v \in \fg_{\leq 0} .
\end{equation}
Using (\ref{eq:2.14}) we obtain
\begin{equation}
    \label{eq:7.2}
    [{J^{\{ a \}}}_\lambda \Phi_{u}]  =  (-1)^{p(a)} \Phi_{[a,u]},\, a \in \fg^{f}_{0},\, u \in \fg_{1/2} .
\end{equation}
From (\ref{eq:7.1}) and (\ref{eq:7.2}) we deduce, by making use of the non-commutative Wick formula, for $a \in \fg^{f}_{0}$, $i \in S_{0}$, $j \in S_{1/2}$ :
\begin{displaymath}
  [{J^{\{ a \}}}_{\lambda} : J^{(u^{i})} J^{(u_{i})} :]  =\,  : J^{([a,u^{i}])} J^{(u_{i})} : + (-1)^{p(a)p(i)} : J^{(u^{i})} J^{([a,u_{i}])} :
\end{displaymath}
\begin{displaymath}
+ (-1)^{p(i)} \lambda J^{([u_{i}, [u^{i},a]])} + \lambda (B_{0} (a,u^{i}) J^{(u_{i})} + (-1)^{p(i)} B_{0} (a,u_{i}) J^{(u^{i})} )
%\end{displaymath}
%\begin{displaymath}
+ \frac{1}{2} \lambda^{2} B_{0} (a , [u^{i} , u_{i}]) ,
\end{displaymath}
\begin{displaymath}
[{J^{\{ a \}}}_{\lambda} : \Phi^{j} \partial \Phi_{j} : ] \ = \ (-1)^{p (a) (p (j) \ + \ 1)}\lambda : \Phi^{j} \Phi_{[a, u_{j}]} : \ + \ \frac{1}{2} (-1)^{p(a)} \lambda^{2} \ \langle [a, u^{j}] , u_{j}  \rangle^{\ne} , 
\end{displaymath}
\begin{displaymath}
[{J^{\{ a \}}}_{\lambda} : \Phi^{j} J^{([f, u_{j}])} : ]  =  (-1)^{p(a)} \ :  \Phi_{[a, u^{j}]} J^{([f, u_{j}])}  : \ + (-1)^{p(a)p(j)} \ :  \Phi^{j} J^{([a, [f, u_{j}]])}  : .
\end{displaymath}
Summing up both sides of the first formula over $i \in S_{0}$, the second and third formula over $j \in S_{1/2}$, we obtain the following three formulas
for  $a \in \fg^{f}_{0}$ :
\begin{equation}
    \label{eq:7.3}
    \underset{i \in S_{0}}{\sum}  [{J^{\{ a \}}}_{\lambda}  :  J^{(u^{i})} J^{(u_{i})}  : ]  =  2 (k + h^{\vee}) \lambda J^{(a)} ,
\end{equation}
\begin{equation}
    \label{eq:7.4}
    \underset{j \in S_{1/2}}{\sum}  [{a^{\ne}} _\lambda  : \Phi^{j}  \partial \Phi_{j}  :  ]  =   -2 \lambda a^{\ne} ,
\end{equation}
where $a^{\ne}$ is the secind term on the right in (\ref{eq:2.14}) ,
\begin{equation}
    \label{eq:7.5}
    \underset{j \in S_{1/2}}{\sum}  [{J^{\{ a \}}}_{\lambda} (-1)^{p(j)}
       :  \Phi^{j} J^{([f, u_{j}])}  : ]  =  0 .
\end{equation}
Using (\ref{eq:7.1}), we obtain 
\begin{equation}
    \label{eq:7.6}
    [{J^{\{ a \}}}_{\lambda} J^{(f)}]  =  0  =  [{J^{\{ a \}}}_{\lambda} J^{(x)}] \ \text{for} \, a \in \fg^{f}_{0}.
\end{equation}
Using Proposition \ref{prop:3.2},
%\begin{equation}
 %   \label{eq:5.7}
  %        [\rho_{> 0} , a]  =  0 \, \text{and} \, \kappa_{0} (\rho_{> 0} , a)
   %       =  0 \, \text{for} \, a \in \fg_{0},
%\end{equation}
we obtain from (\ref{eq:7.1}) :
\begin{equation}
    \label{eq:7.7}
    [{J^{\{ a \}}}_\lambda  J^{(\rho_{> 0})}]  =  \lambda (k + h^{\vee})(a |\rho_{> 0}) .
\end{equation}

Now we can complete the proof of the lemma.  Formula (a) is straightforward by the
discussion in Section \ref{sec:3}, cf. (\ref{eq:3.4}).  Below we shall prove (b).
We have for $a \in \fg^{f}_{0}$ :
\begin{displaymath}
[{J^{\{ a \}}}_{\lambda} J^{\{ f \}}]  =  [{J^{\{ a \}}}_{\lambda} J^{(f)}]
%\end{displaymath}
%\begin{displaymath}
+ \ \underset{i \in S_{1/2}}{\sum} (-1)^{p(i)}  [{J^{\{ a \}}}_{\lambda}  :  \Phi^{i} J^{([f, u_{i}])} : ]
\end{displaymath}
\begin{displaymath}
- \frac{1}{2} \underset{i \in S_{0}}{\sum} [{J^{\{ a \}}}_{\lambda} : J^{(u^{i})} J^{(u_{i})} : ] - (k + h^{\vee}) [{J^{\{ a \}}}_{\lambda} \partial J^{(x)}] 
\end{displaymath}
\begin{displaymath}
+ [{J^{\{ a \}}}_{\lambda} \ \partial J^{(\rho_{> 0})}] + \frac{1}{2} (k + h^{\vee}) \underset{i \in S_{1/2}}{\sum} [{J^{\{ a \}}}_{\lambda}  : \Phi^{i} \partial \Phi_{i} :].
\end{displaymath}
The first and the fourth terms on the right are equal to $0$ by (\ref{eq:7.6}), and the second term equals $0$ by (\ref{eq:7.5}).  The third term equals $-(k + h^{\vee}) \lambda J^{(a)}$ by (\ref{eq:7.3}).  The fifth term equals $\lambda^{2} (k + h^{\vee}) (a | \rho_{> 0})$ by (\ref{eq:7.7}).  The sixth term equals $-(k + h^{\vee}) \lambda a^{\ne}$ by (\ref{eq:7.4}).  Thus, we have :
\begin{displaymath}
[{J^{{\{ f \}}}}_{\lambda} J^{\{ a \}}]  =  - [{J^{\{ a \}}}_{- \partial - \lambda} J^{\{ f \}}]
%\end{displaymath}
%\begin{displaymath}
  =  - (k + h^\vee) (\partial + \lambda) J^{\{ a \}}  -
  \lambda^2 (k + h^\vee) (a | \rho_{> 0})
\end{displaymath}
\begin{displaymath}
= \ - (k + h^{\vee}) ((\partial + \lambda) J^{\{ a \}}  + \lambda^{2} (a | \rho_{> 0}) ),
\end{displaymath}
proving formula (b).
\end{proof}
\end{lemma}
\begin{remark}
  \label{rem:7.1}
  The same proof shows that Lemma \ref{lem:7.1} holds for arbitrary grading
  (\ref{eq:1.1}), satisfying  (\ref{eq:1.5}), if we replace the coefficient of
  $\lambda ^2$ by $(\rho_{>0}|a)-(k+h^\vee)(x|a)$.
\end{remark}

%Before stating the next lemma. we introduce a generalization of the Sugawara
%construction (see e.g. \cite{K2}) to the case of an arbitrary finite-dimensional Lie superalgebra $\fg$ with a non-degenerate supersymmetric bilinear form $B$.
%Let $V^B(\fg)$ be the associated universal affine vertex algebra. Let
%$\Omega=\sum_i u^iu_i$ be the Casimir element of $\fg$, and let
%$\fg=\sum_j\fg^j$
%be the eigenspace decomposition of $\fg$ with respect to $\Omega$, where
%$\fg^j$ is the eigenspace, attached to the eigenvalue $2h_j^\vee$. Let
%$2H_j^\vee$ be the image of $\Omega$ in $\End \,\fg^j$. Let $k\in\FF$ and let
%$N_j=\frac{1}{k+h_j^\vee}(H_j^\vee-h_j^\vee I_{\fg^j})$ (we assume that $k\neq
%-h_j^\vee$).
%Let $\{a_i\}$ and $\{a^i\}$ be dual bases of $\fg^j$, and consider the following elements of the vertex algebra $V^B(\fg)$:
%\begin{displaymath}
 % L^j=\frac{1}{(k+h_j^\vee)(1+N_j)}\sum_i a^ia_i,\,\,\,\, L_{Sug}=\sum_j L^j.
%  \end{displaymath}
%Note that the operator $N_j$ is nilpotent on $\fg^j$, hence locally nilpotent on $V^B(\fg)$, therefore the generalized Sugawara element $L_{Sug}$ of $V^B(\fg)$
%is well defined. The following proposition is proved in the same way as for the usual Sugawara operator (see e.g. \cite{K2}).
%\begin{proposition}
%  \label{prop:7.1}
%The element $L_{Sug}$ is a conformal vector of the vertex algebra $V^B(\fg)$, and all elements $a\in \fg$ have conformal weight $1$ with respect to it.
%\end{proposition}

\begin{lemma}
  \label{lem:7.2}
  Assume that condition (i) holds and let $M$ be an element of confirmal weight 2 (with respect to $L$), which lies in the subalgebra of $\C^{k} (\fg , x, f)$, generated by the elements $J^{\{ a \}}$, $a \in \fg^{f}_{0}$.    Suppose that
  \begin{equation}
    \label{eq:7.8}
   [{J^{\{ a \}}}_{\lambda} M] = 0\,\, \hbox{for all}\,\,\, a \in \fg^{f}_{0}.
  \end{equation}
  Then $M  =  0$.
  \begin{proof}    By the condition (i),
   % $\fg^{f}_{0}$ is a direct sum of its center and some simple components.  Hence,
    the subalgebra of $\C^{k} (\fg , x, f)$ generated by the elements $J^{\{ a \}} , a \in \fg^{f}_{0}$, is the affine vertex algebra $V^{B_{1/2}} (\fg^{f}_{0})$ (see (\ref{eq:2.15}), (\ref{eq:2.16})), associated to the Lie superalgebra
    $\fg^f_0$ and a non-degenerative bilinear form on it, for all, but finitely many values of $k$. It follows that $\fg_0^f$ is a direct sum of its center, some basic Lie superalgebras and Lie superalgebras $gl(n|n), n\geq 1$. Therefore, for all, but finitely many values of $k$, this vertex algebra carries the
    %generalized
    Sugawara element, with respect to which all elements $J^{\{ a \}}$, $a \in \fg^{f}_{0}$, have conformal weight 1 (see e.g. \cite{K2}); for $gl(n|n)$ we use the modification of the Sugawara construction, discussed in Appendix B. Hence, the center of $V^{B_{1/2}} (\fg^{f}_{0})$ consists of the multiples of the vacuum vector.
    Since condition (\ref{eq:7.8}) implies that M is a central element of $V^{B_{1/2}} (\fg^{f}_{0})$, we conclude that $M = 0$ for all, but finitely many $k$, hence for all $k$.
  \end{proof}
\end{lemma}

Now we can complete the proof of Theorem \ref{th:3.2}.
    Since L has conformed weight 2 (hereafter conformal weight is meant with respect to L), by [KW], Theorem 4.1(a), L is $d_{(0)}$-equivalent to an element of the form 
        \begin{equation}
        \label{eq:7.9}
        J^{\{ f^{'} \}} \ = \ J^{(f^{'})} \ + \ M_{1} ,
    \end{equation}
        where $f^{'} \in \fg^{f}_{-1}$ and $M_{1}$ lies in the subalgebra of
        $\C^k(\fg , x, f)$ generated by the $J^{(a)}$ with
        $a \in \fg_{0} + \fg_{-1/2}$ and the $\Phi_{i}, i \in S_{1/2}$ .  Let $v \in \fg^{f}$ and let $J^{\{ v \}} \in \overline{\C}^{k} (\fg , x, f)$ be an element, given by \cite{KW}, Theorem 4.1(a).  Since $J^{\{ f^{'} \}}$ is $d_{(0)}$-equivalent to L, we have in $W^{k} (\fg , x, f)$
        \begin{equation}
        \label{eq:7.10}
        [{J^{\{ f^{'} \}}}_{\lambda} J^{\{ v \}}] \ = \ \partial J^{\{ v \}} \ + \ O (\lambda).
    \end{equation}
        But, due to (\ref{eq:2.8}), equation (\ref{eq:7.10}) is an equality in $\overline{\C}^{k} (\fg , x, f)$. It is easy to conclude from (\ref{eq:7.9}), using (\ref{eq:2.10}) (which can be used since $f^{'} \in \fg_{-1}$ and $v \in \fg_{\leq 0}$), that we have the following equality in $\overline{\C}^{k} (\fg , x, f)$ :
    \begin{equation}
    \label{eq:7.11}
    [{J^{\{ f^{'} \}}}_{\lambda} J^{\{ v \}}] \ = \ J^{\{ [f^{'} , v] \}} + A + B + O (\lambda),
\end{equation}
where $A$ is a linear combination of elements of the form $\sum \ :  \partial^{i_{1}} J^{(v_{1})} \dots \partial^{i_{n}} J^{(v_{n})}:$
with $n \geq 2$ or $i_{1} \ + \dots + i_{n} \geq 1$, and $B$ is a linear combination of normally ordered products, involving the $\Phi_{i}$.
Comparing (\ref{eq:7.10}) and (\ref{eq:7.11}), we conclude that 
\begin{displaymath}
[f^{'} , \fg^{f}]  =  0.
\end{displaymath}
Hence, by condition (ii), we have 
\begin{displaymath}
f^{'}  = \gamma f \ \text{for some} \ \gamma \in \mathbb{F} .
\end{displaymath}
Therefore, L - $\gamma J^{(f)}$ is a sum of normally ordered products of elements of $\overline{\C}^{k} (\fg , x, f)$ of conformal weight $< 2$. 
It follows that, for $J^{\{ f \}}$ from Theorem \ref{th:3.1}, we have 
\begin{equation}
    \label{7.12}
    L \ - \ \gamma J^{\{ f \}} \ = \ M ,
\end{equation}
where M is a $d_{(0)}$-closed element of conformal weight 1. Hence, by \cite{KW}, Theorem 4.1(b) , M is a linear combination of elements of the form : $J^{\{ a \}} J^{\{ b \}}$ : and $\partial J^{\{ c \}}$ , where $a, \ b, \ c \ \in \ \fg^{f}_{0}$.  By Lemma \ref{lem:7.1}, M satisfies (\ref{eq:7.8}), hence, by Lemma \ref{lem:7.2}, $M = 0$, and we have
\begin{equation}
    \label{eq:7.13}
    L  =  \gamma J^{\{ f \}}.
\end{equation}

In order to complete the proof of Theorem \ref{th:3.2}, it remains to show that
\begin{equation}
    \label{eq:7.14}
    \gamma \ = \ - \frac{1}{k + h^{\vee}} .
\end{equation}
For that we use the following formula, which can be deduced from the discussion of properties of L in Section 3 :
\begin{equation}
    \label{eq:7.15}
          [L_{\lambda} J^{(v)}] \ = \ (\partial + \Delta_v \lambda) J^{(v)} \ +
            \ \lambda^{2} (\rho_{> 0} | v),\,\, v \in \fg,
\end{equation}
where $\Delta_{v}$ is the conformal weight of $J^{(v)}$.
Using formula (\ref{eq:7.15}), we obtain 
\begin{equation}
    \label{eq:7.16}
    [L_{\lambda} \gamma J^{\{ f \}}] \ = \ (\partial + 2 \lambda) \gamma J^{\{ f \}} \ + \ \lambda^{3} \big(- \gamma (k + h^\vee) \frac{c (\fg, x, k)}{12}+\gamma \beta \big).
\end{equation}
where
\begin{equation*}
  \beta = \ (\rho_{> 0} | \rho_{> 0}) - (\rho  |  \rho) + \frac{1}{6}
  (\rho_{1/2}  | \rho_{> 0})
+  \frac{1}{24}  \str_{\fg_{0} \oplus \fg_{1/2}} \Omega_0 .
\end{equation*}
Using (\ref{eq:7.13}) and comparing (\ref{eq:7.16}) with 
\begin{displaymath}
[L_\lambda L] \ = \ (\partial + 2 \lambda) L \ + \ \frac{\lambda^{3}}{12} \ c (\fg, x, k),
\end{displaymath}
where $c (\fg, x, k)$,  given by (\ref{eq:3.3}), is non-zero for generic k, we obtain (\ref{eq:7.14})  (and also that $\beta = 0$).
This completes the proof of Theorem \ref{th:3.2}.

\section{ Appendix B. Theorem 3.2 for $\fg = gl(n|n)$.}
\label{sec:8}

Let $\fg$ be a finite-dimensional Lie superalgebra over $\mathbb{F}$ with an even invariant supersymmetric bilinear form $(. | .)$. In order to apply the Sugawara construction, we need two properties: 

(i) the bilinear form $(. | .)$ is non-degenerate, so that we can choose dual bases $\{ u_i \}$ and $\{ u^i \}$ of $\fg$ with respect to this form and construct the Casimir operator
\begin{displaymath}
\Omega \ = \ \underset{i}{\sum}\ u^i u_i \ \in U (\fg);
\end{displaymath}

(ii) the Casimir operator $\Omega$ acts on $\fg$ as a scalar (which we denoted by $2 h^{\vee}$). 

In this Appendix we consider $\fg = gl (n|n)$ with the bilinear form
$(a|b) = \str\, ab$. The property (i) holds, but (ii) fails. However, we will show that the Sugawara operator $L^{\fg}$, appearing in the formula (\ref{eq:3.2}) for L can be modified, so that the resulting modified $L$ is a
Virasoro vector, which satisfies a modified Theorem \ref{th:3.2} with the modified $J^{\{ f \}}$.

%\begin{proof}
Let $I$ be the identity matrix in $\fg$ and let
\begin{equation}
    \label{eq:8.1}
    \omega \ = \ \underset{i}{\sum} \ : u^i u_i : \ \in V^{k} (\fg).
\end{equation}
The following formulas are obtained by straightforward computations, where $a \in \fg$ :
\begin{equation}
    \label{eq:8.2}
    \Omega (a) \ = \ - 2 (a|I) I ;
\end{equation}
\begin{equation}
    \label{eq:8.3}
    [a_{\lambda} \omega] \ = \ 2 \lambda k a - 2 \lambda (a|I) I ;
\end{equation}
\begin{equation}
    \label{eq:8.4}
    [a_{\lambda} I] \ = \ \lambda k (a|I) ;
\end{equation}
\begin{equation}
    \label{eq:8.5}
    [a_{\lambda} : I^2 :] \ = \ 2 \lambda k (a|I) I .
\end{equation}

Introduce the modified Sugawara operator
\begin{equation}
    \label{eq:8.6}
    L^{\fg} \ = \ \frac{1}{2 k} \omega \ + \ \frac{1}{2 k^2} \ :  I^2: .
\end{equation}
It is straightforward to deduce from (\ref{eq:8.3}) - (\ref{eq:8.5}) the following two formulas:
\begin{equation}
    \label{eq:8.7}
          [a_{\lambda} L^{\fg}] \ = \ \lambda a ,\,\, \text{hence}\,\,
          [{L^{\fg}}_{\lambda} a] \ = \ (\partial + \lambda)a ;
\end{equation}
\begin{equation}
    \label{eq:8.8}
    [{L^{\fg}}_\lambda  L^{\fg}] \ = \ (\partial + 2 \lambda) L^{\fg} .
\end{equation}
Hence, we have the following proposition.
\begin{proposition}
 \label {prop:8.1}.
  The element $L^{\fg}$ defined by (\ref{eq:8.6}) is a Virasoro vector of $V^k (\fg)$ with central charge 0, for which $a \in \fg $ have conformal weight 1. Hence, $V^k (\fg)$ is a conformal vertex algebra of CFT type for all $k \neq 0$.
%\end{proof}
\end{proposition}

Finally, we have the following version of Theorem \ref{th:3.2} for
$\fg = gl(n|n)$.
\begin{theorem}
    \label{th:8.1}
    Let $L^{\fg}$ be defined by (\ref{eq:8.6}), let $L$ be defined by (\ref{eq:3.2}) with this $L^{\fg}$, and let $\Tilde{J}^{\{ f \}}$ be the element, defined in Theorem \ref{eq:3.1} for $h^{\vee} = 0$. Then the element
    %the same proof as that of \ref{th:3.1} shows that this element is $d_{0)}$-closed. Let 
        \begin{displaymath}
    J^{\{ f \}} \ = \ \Tilde{J}^{\{ f \}}  - \frac{1}{2 k} : I^{2} : 
    \end{displaymath}
  is $d_{(0)}$-closed
         and $L  =  - \frac{1}{k} J^{\{ f \}}$ is an energy-momentum vector of $W^{k} (\fg , x , f)$ with central charge given by formula (\ref{eq:3.3}) with $h^{\vee} = 0$.
        \begin{proof}
          By the same proof as that of Theorem \ref{th:3.1}, the element
          $\Tilde{J}^{\{ f \}}$ is $d_{(0)}$-closed, and also, by (\ref{eq:2.6}), the element $I$ is $d_{(0)}$- closed, hence the same holds for $: I^{2} :$.
          Hence the element $J^{\{f\}}$ is $d_{(0)}$-closed.
          
          Let $\Tilde{L} \ = \frac{1}{2 k} \omega  +  \partial x + L^{ch} + L^{ne} $ (cf. (\ref{eq:3.2})).
          By the same proof as that of Theorem \ref{th:3.2}, we have: $\Tilde{L} + \frac{1}{k} \Tilde{J}^{\{ f \}}$ is $d_{(0)}$-exact.
          Since this element coincides with $L + \frac{1}{k} J^{\{ f \}}$,
          the theorem is proved.
    \end{proof}
\end{theorem}

\end{document}